\pdfoutput=1
\documentclass[11pt,a4paper,reqno,oneside]{amsart}
\usepackage[utf8]{inputenc}
\usepackage{graphicx}
\usepackage[english]{babel}
\usepackage{caption}
\usepackage{subcaption}
\usepackage[a4paper,width=150mm,top=30mm,bottom=40mm,marginparwidth=2.5cm]{geometry}
\usepackage{fancyhdr}
\usepackage{amsmath, amssymb, amscd, amsthm, young, stmaryrd}
\usepackage[dvipsnames]{xcolor}
\usepackage{mathrsfs, mathdots, tikz, mathabx, enumerate, paralist, scalerel, relsize} 
\usepackage[mathcal]{eucal} 
\usepackage{float}
\usepackage[T1]{fontenc}
\usepackage{tikz-cd}
\usetikzlibrary{matrix, calc, arrows}
\usepackage{pgfplots}
\usepgfplotslibrary{fillbetween}
\pgfplotsset{compat=1.7}
\usepackage{longtable}
\usepackage{color}
\usepackage{url}
\usepackage{pgf}
\usepackage[linktocpage=true]{hyperref}
\hypersetup{
	colorlinks=true,
	linkcolor=[RGB]{0,104,180},
	citecolor=[RGB]{0,104,180},    
	urlcolor=[RGB]{0,104,180},
}
\usepackage[backend=bibtex,style=alphabetic]{biblatex}
\newcommand\blfootnote[1]{%
  \begingroup
  \renewcommand\thefootnote{}\footnote{#1}%
  \addtocounter{footnote}{-1}%
  \endgroup
}

\DeclareMathOperator{\des}{des}
\DeclareMathOperator{\maj}{maj}
\DeclareMathOperator{\Des}{Des}
\DeclareMathOperator{\std}{std}

\DeclareMathOperator{\inv}{inv}

\DeclareMathOperator{\Mat}{Mat}
\DeclareMathOperator{\stat}{stat}

\DeclareMathOperator{\vol}{vol}
\DeclareMathOperator{\Ehr}{Ehr}
\DeclareMathOperator{\Lp}{L}
\DeclareMathOperator{\wt}{wt}
\DeclareMathOperator{\conv}{conv}

\newcommand\cp{\scaleobj{1.7}{\diamond}}
\newcommand\mboxdot{\scaleobj{1.13}{\boxdot}}
\newcommand{\Conv}{\mathop{\scalebox{1.5}{\raisebox{-0.2ex}{$\ast$}}}}

\def\0{\hspace{-.12em}0}
\definecolor{dgreen}{HTML}{03297A}

\definecolor{dgreen}{HTML}{0B3B17}

\providecommand{\keywords}[1]
{
  \small	
  \textbf{\textit{Keywords---}} #1
}

  \theoremstyle{plain}
  \newtheorem{thm}{Theorem}[section]
  \newtheorem*{Atheorem}{Theorem~A}
  \newtheorem*{Btheorem}{Theorem~B}  
  \newtheorem{lem}[thm]{Lemma}
  \newtheorem{prop}[thm]{Proposition}
  \newtheorem{cor}[thm]{Corollary}
  \newtheorem{con}[thm]{Conjecture}

  \theoremstyle{remark}
    
  \newtheorem{rem}[thm]{Remark} 
  \newtheorem{pro}[thm]{Problem}
  \newtheorem*{acknowledgements}{Acknowledgements}
  
  \theoremstyle{definition}
  \newtheorem{dfn}[thm]{Definition}
  \newtheorem{exa}[thm]{Example}
  \newtheorem{ass}[thm]{Assumption}

  \numberwithin{thm}{section}
  \numberwithin{equation}{section}

\bibliography{references}

\author[E.~Tielker]{Elena Tielker}
\address{Fakult\"at f\"ur Mathematik, Universit\"at Bielefeld, D-33501 Bielefeld, Germany
}
\email{etielker@math.uni-bielefeld.de}

 \title{Weighted Ehrhart series and a type-$\mathsf{B}$ analogue of a formula of MacMahon} 
 \subjclass{05A05, 52B11}
 \keywords{descents, signed multiset permutations, coloured multiset permutations, Ehrhart theory, $h^*$-polynomials, $q$-analogues}

\begin{document}
\begin{abstract}We present a formula for a generalisation of the Eulerian polynomial, namely the generating polynomial of the joint distribution of major index and descent statistic over the set of signed multiset permutations. It has a description in terms of the ${h^*\text{-polynomial}}$ of a certain polytope. Moreover, we associate a family of polytopes to (generalised) Eulerian polynomials of types~$\mathsf{A}$ and~$\mathsf{B}$. Using this connection, properties of the generalised Eulerian numbers of types~$\mathsf{A}$ and~$\mathsf{B}$, such as palindromicity and unimodality, are reflected in certain properties of the associated polytope.
We also present results on generalising the connection between descent polynomials and polytopes to coloured (multiset) permutations.
  \end{abstract} 
 
\maketitle
\thispagestyle{empty}

\tableofcontents

\section{Introduction}\blfootnote{An extended abstract \cite{Tielker} of this paper has been accepted in \mbox{FPSAC} 2022.}

This paper develops a relationship between some well-known permutation statistics and lattice point enumeration in polytopes.
A classical instance of this relation is 
\begin{equation}\label{MacMahonFormulaCube}
\frac{d_{S_n}(t)}{(1-t)^{n+1}} = \sum_{k\geq 0} (k+1)^n t^k,
\end{equation}
where $d_{S_n}(t)$ is the $n$th Eulerian polynomial.
This polynomial is described in terms of descents on the elements in the symmetric group $S_n$ (cf.\ Section~\ref{section:submultisetperm} for $\eta=(1,\dots ,1)$), while $(k+1)^n$ is the number of lattice points in the $k$th dilate of the $n$-dimensional unit cube, i.e.\ its generating function is the Ehrhart series of the cube.

The left hand side of \eqref{MacMahonFormulaCube} can be generalised in several directions:
\begin{itemize}
\item[(i)] by repeating letters we obtain multiset permutations,
\item[(ii)] instead of $S_n$ we may consider descents over the hyperoctahedral group $B_n$, a Coxeter group of type~$\mathsf{B}$, or, yet more generally, coloured permutations,
\item[(iii)] in addition, by taking another statistic, the major index, into account we obtain a refinement of the number of descents.
\end{itemize}
In this paper we explore all three directions and develop a unifying perspective through the lens of generating functions:
the generalised polynomial (in the sense of (i)-(iii)) thus obtained can be interpreted in terms of $q$-analogues of Ehrhart series of certain polytopes which reflect the generalisations made on the permutation side.
For example, we interpret the generalisation of the Eulerian polynomial in (i) and (iii), which is known as a formula of MacMahon~\cite{MacMahon}, as follows: 
\begin{Atheorem}[MacMahon's formula of type~$\mathsf{A}$]\label{ThmA}
The generating polynomial of the joint distribution of the major index and descent statistic over the set of multiset permutations as in 
\begin{equation}\label{MacMahonFormula}
\frac{\sum_{w\in S_{\eta}} q^{\maj(w)} t^{\des(w)}}{\prod_{i=0}^n (1-q^it)} = \sum_{k\geq 0} \left( \prod_{i=1}^r \binom{k+\eta_i}{\eta_i}_q \right) t^k \in \mathbb{Q}(q,t),
\end{equation}
where $\eta=(\eta_1,\dots,\eta_r)$ is a composition of $n$ and $\binom{n}{k}_q$ the $q$-binomial coefficient, is the numerator of a $q$-analogue of the Ehrhart series of products of standard  simplices.
\end{Atheorem}
Theorem~\ref{MacMahonA} is a refined version of \eqref{MacMahonFormula} which provides a detailed description of the $q$-analogue of the Ehrhart series.
In the special case where $\eta = (1,\dots,1)$ and $q=1$, we obtain \eqref{MacMahonFormulaCube}, thus the Eulerian polynomial on the left hand side and the Ehrhart series of products of one-dimensional simplices, i.e.\ of the cube, on the right hand side.
Since $S_n$ is a Coxeter group of type~$\mathsf{A}$, we refer to \eqref{MacMahonFormula} as MacMahon's formula of type~$\mathsf{A}$.
Summing up, the type-$\mathsf{A}$ descent polynomials correspond to (standard) simplices.\\
The natural extension of \eqref{MacMahonFormula} to a type-$\mathsf{B}$ descent polynomial (viz.\ on $B_n$ or more general on its generalisation $B_{\eta}$, the set of signed multiset permutations) is via the polytope side. This is what we develop in the current paper. 
More precisely, we count (weighted) lattice points in cross polytopes, which can be seen as signed analogues of simplices.
Our main result verifies this relationship and further proves a refinement of this type-$\mathsf{B}$ extension which we therefore call MacMahon's formula of type~$\mathsf{B}$:
for this reason we give new definitions of the statistics major index and descent on the set of signed multiset permutations --- and therefore define a new generalisation of the Eulerian numbers of type~$\mathsf{B}$ --- and construct weight functions on the integer points in products of cross polytopes.
Our main result extends \eqref{MacMahonFormulaCube} in all three directions mentioned above at the same time:
\begin{Btheorem}[MacMahon's formula of type~$\mathsf{B}$]\label{ThmB}
The generating polynomial of the joint distribution of the major index and descent statistic over the set of signed multiset permutations as in 
\begin{equation}\label{MacMahonFormulaTypeB}
\frac{\sum_{w\in B_{\eta}}q^{\maj(w)}t^{\des(w)}}{\prod_{i=0}^n(1-q^i t)} = \sum_{k\geq 0} \left( \prod_{i=1}^r \sum_{j= 0}^{\eta_i} \left( q^{\frac{j(j-1)}{2}} \binom{\eta_i}{j}_q \binom{k-j+\eta_i}{\eta_i}_q \right)\right) t^k \in \mathbb{Q}(q,t)
\end{equation}
is the numerator of a $q$-analogue of the Ehrhart series of products of cross polytopes.
\end{Btheorem}

\begin{figure}
\centering
\includegraphics[width=0.33\textwidth]{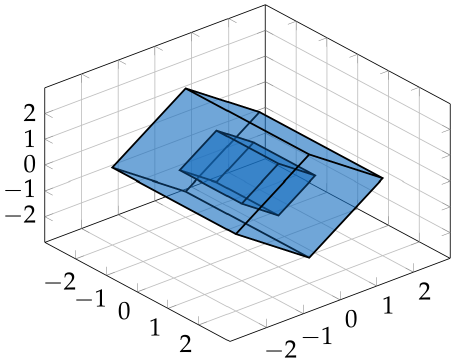}
\caption{The first and second dilate of the product of a one-dimensional and a two-dimensional cross polytope.}
\label{fig:dreidim}
\end{figure}

Considering MacMahon's formulae of types~$\mathsf{A}$ and~$\mathsf{B}$ from the polytope side, we give explicit descriptions of (a $q$-analogue of) the $h^*$-polynomials of two families of polytopes, namely products of standard simplices and products of cross polytopes, in Theorems~\ref{MacMahonA} and~\ref{MacMahonB}, respectively.
Moreover, the Ehrhart series \eqref{MacMahonFormulaTypeB} give rise to a so-called $q$-Ehrhart polynomial, which is a polynomial in the $q$-binomial coefficient $\binom{k}{1}_q$ and was introduced by Chapoton \cite{Chapoton}.
For example, for $\eta=(1,2)$ the generating function on the right hand side of \eqref{MacMahonFormulaTypeB} encodes the numbers of weighted lattice points in the $k$th dilates of the product of a one-dimensional and a two-dimensional cross polytope. For $k\in \{1,2\}$ the polytope is shown in Figure~\ref{fig:dreidim}. \\
A special case of Theorem~\hyperref[MacMahonFormulaTypeB]{B}, viz.\ Corollary~\ref{Cor_hypercubein0}, interprets the Eulerian numbers of type~$\mathsf{B}$ as the coefficients of the $h^*$-polynomial of a polytope and, more generally, the joint distribution of major index and descent over the hyperoctahedral group $B_n$ as a $q$-analogue of this $h^*$-polynomial. 
We refer to Section~\ref{section:MacMahonTypeB} for further details.\\
Using this connection between the Ehrhart series of certain polytopes and permutation statistics, palindromicity and unimodality of several generalisations of the Eulerian numbers can be seen via the polytope-theoretic counterpart, namely Gorenstein and anti-blocking polytopes; cf.\ Section~\ref{section:Properties}.\\
In Section~\ref{section:Generalisation}, we define coloured multiset permutations as generalisations of (signed) multiset permutations and a descent statistic on them 
which is equidistributed over the group of coloured permutations to descent statistics defined by \cite{BeckBraun} and \cite{Steingrimsson}. 
We discuss a potential candidate for a polytope corresponding to this set of permutations and further present partial results on the connection between Ehrhart series and descent statistics.

\subsection{Notation}
We write $\mathbb{N}= \{1,2,\dots \}$ and $\mathbb{N}_0= \{0,1,2,\dots \}$. For $n\in \mathbb{N}$ we write $[n]=\{ 1, \dots, n \}$ and $[n]_0=\{ 0, \dots, n \}$. 
For $q$ a variable and $k\in \mathbb{N}_0$ we denote by 
\begin{align*}
\binom{n}{k}_q = \frac{(1-q^n)\cdots (1-q^{n-(k-1)})}{(1-q) \cdots (1-q^k)} %\prod_{i=1}^k \frac{q^{n-k+i}-1}{q^i-1} \in \mathbb{Z}[q]
\end{align*}
the \emph{$q$-binomial coefficient} and further write
\begin{align*}
\binom{n}{k}_q & = \frac{[n]_q!}{[n-k]_q! [k]_q!}, \\
\text{where } [k]_q! &= [k]_q \cdots [2]_q [1]_q \\
\text{and } [k]_q &= \frac{1-q^k}{1-q}.
\end{align*}
We denote by $|S|$ the cardinality of a set $S$.
We write $\Mat_n(\mathbb{Z})$ for the set of $n\times n$-matrices over $\mathbb{Z}$.
Throughout, let $\eta=(\eta_1,\dots ,\eta_r)$ denote a composition of $n\in \mathbb{N}$ into $r$ parts, i.e.\ $\eta_i\in \mathbb{N}$ for all $i$ and $\eta_1 + \dots + \eta_r = n$.

\section{Preliminaries}\label{section:Preliminaries}

\subsection{Permutations statistics}\label{subsection:PermStat}

The numerator of the rational functions in \eqref{MacMahonFormula} and~\eqref{MacMahonFormulaTypeB} are described in terms of permutations statistics. In the following we recall the relevant definitions of statistics on multiset permutations and define new ones on signed multiset permutations.

\subsubsection{Multiset permutations} \label{section:submultisetperm}
A multiset permutation $w$ is a rearrangement of the letters of the multiset 
\begin{align*}
\{\{ \underbrace{1,\dots,1}_{\eta_1},\underbrace{2,\dots,2}_{\eta_2},\dots,\underbrace{r,\dots,r}_{\eta_r} \}\}.
\end{align*}
We write $w = w_1\cdots w_n$ (using the one-line notation) for such a permutation and denote by $S_{\eta}$ the set of all permutations of the multiset given by a composition $\eta =(\eta_1,\dots ,\eta_r)$ of $n$. 
The \emph{descent set} is defined to be $\Des(w) = \{i\in [n-1] : w_i > w_{i+1}\}$. Elements in $\Des(w)$ are called \emph{descents of $w$}. The \emph{major index} and the \emph{descent} statistic are 
\begin{align*}
\maj(w) = \sum_{i\in \Des(w)} i\quad \text{and}\quad \des(w) = |\Des(w)|.
\end{align*}
If, for example, $\eta = (2,3)$, then $w = 22121$ is a permutation of the corresponding multiset $\{\{1,1,2,2,2\}\}$. Here, $\Des(w) = \{2,4\}$ and therefore $\maj(w) = 6$ and $\des(w) = 2$.

\begin{dfn}\label{dfn:descentPolynomialS_eta}
We denote by 
\begin{align*}
d_{S_{\eta}}(t) := \sum_{w\in S_{\eta}} t^{\des(w)} \in \mathbb{Z}[t]
\end{align*}
the \emph{descent polynomial of $S_{\eta}$} (which we also call the \emph{generalised Eulerian polynomial} (\emph{of type $\mathsf{A}$})).
Therefore, we call its coefficients \emph{generalised Eulerian numbers} (\emph{of type $\mathsf{A}$}).
\end{dfn}
Note that for $\eta = (1,\dots,1)$ we have $S_{\eta} = S_n$ and the coefficients of $d_{S_n}$ are the Eulerian numbers (of type $\mathsf{A}$).

\begin{dfn}\label{dfn:Carlitz-q-Eulerian-polynomialS_eta}
The \emph{generalised Carlitz's $q$-Eulerian polynomial} is the bivariate generating polynomial for the major index and descent statistic over the set of multiset permutations, namely
\begin{align*}
C_{S_{\eta}}(q,t) := \sum_{w\in S_{\eta}} q^{\maj(w)} t^{\des(w)} \in \mathbb{Z}[q,t].
\end{align*}
\end{dfn} 
\noindent For instance, for $\eta = (2,3)$ we have $$C_{S_{(2,3)}}(q,t)= (q^6 + q^5 + q^4) t^2 + (q^4 + 2 q^3 + 2 q^2 + q )t + 1.$$
Note that $C_{S_{\eta}}$ is indeed a generalisation of the well-known Carlitz's $q$-Eulerian polynomial; see~\cite{Carlitz54, Carlitz75}.
Further, we remark that $C_{S_{\eta}}$ only depends on the partition of $n$. We define it --- and analogously $C_{B_{\eta}}$ in Definition~\ref{dfn:Carlitz-q-Eulerian-polynomialB_eta} --- for the composition anyway.

\subsubsection{Signed multiset permutations} \label{section:subsignedmultisetperm}

In the following we introduce signed multiset permutations and give definitions of the  major index and descent statistic generalising those discussed in Section~\ref{section:submultisetperm}.\\
Recall (e.g.,~from \cite[Chapter~8.1]{BjBr}) that \emph{signed permutations} are obtained from permutations $w=w_1\cdots w_n\in S_n$ (in one-line notation), where each letter $w_i$ is independently equipped with a sign $\pm 1$. We denote by $B_n$ the set of signed permutations on the letters $1,\dots ,n$.

Similarly, we obtain the set of \emph{signed multiset permutations} $B_{\eta}$ from the set of multiset permutations $S_{\eta}$ by `adding signs': 
more precisely, the elements of $B_{\eta}$ are given by a multiset permutation $w\in S_{\eta}$ and $\epsilon\colon [n] \to \{\pm 1 \}$, a sign vector which attaches every $i$ (or $w_i$) with a positive or negative sign. It is sometimes useful to write an element of $B_{\eta}$ as a pair $w^{\epsilon}:=(w,\epsilon)$, where $w\in S_{\eta}$ and $\epsilon\colon [n] \to \{\pm 1 \}$ encodes the signs appearing in $w^{\epsilon}$.\\
In one-line notation, we write $\bar{i}$ instead of $-i$. 
For example, for $\eta=(2)$ we abbreviate the set of signed multiset permutations to $$B_{\eta}=\{(11,(1,1)), (11,(1,-1)), (11,(-1,1)), (11,(-1,-1))\} = \{11, 1\bar{1}, \bar{1}1, \bar{1}\bar{1}\}.$$
For the definition of a descent set of an element $w\in B_{\eta}$ we need a notion of standardisation.
We use the map $\std\colon S_{\eta}\to S_n$, which is known for multiset permutations (cf.~\cite[Chapter 1]{Stanley}), defined as follows: 
for an element $w\in S_{\eta}$ we obtain $\std(w)\in S_n$ by substituting the $\eta_1$ $1$s from left to right with $1,\dots, \eta_1$, the $\eta_2$ $2$s from left to right with $\eta_1+1,\dots,\eta_1+\eta_2$ and so on.
We extend this \emph{standardisation} to signed multiset permutations:
\begin{align*}
 B_{\eta} &\to B_n, \\
(w,\epsilon) &\mapsto (\std(w), \epsilon).
\end{align*}
We denote both the standardisation on $S_{\eta}$ and the one on $B_{\eta}$ by $\std$.
For instance, $\std(\bar{2}\bar{2}12\bar{1})= \bar{3}\bar{4}15\bar{2}$. 

\begin{dfn}\label{dfn:descent_set-B_eta}
We define the \emph{descent set} of a signed multiset permutation $w^{\epsilon}\in B_{\eta}$ to be
\begin{align*}
\Des(w^{\epsilon}) := 
\{ i\in [n-1]_0 : \std(w^{\epsilon})_i > \std(w^{\epsilon})_{i+1} \},
\end{align*}
where $\std(w^{\epsilon})_0 := 0$.
In other words, 
\begin{align*}
\Des(w^{\epsilon}) = 
\{ i\in [n-1]_0 :\ & \epsilon(i)=\epsilon(i+1)=1 \text{ and } w_i > w_{i+1}, \\
	\text{or } & \epsilon(i)=\epsilon(i+1)=-1 \text{ and } w_i \leq w_{i+1}, \\
	\text{ or } & \epsilon(i)=1 \text{ and } \epsilon(i+1)=-1 \},
\end{align*}
where $w_0 :=0$ and $\epsilon(0) := 1$. 
In particular, $0\in \Des(w^{\epsilon})$ if and only if $\epsilon(1)=-1$.
\end{dfn}
\noindent Note that on elements in $B_n$ our definition of the descent set coincides with the Coxeter-theoretic one; see \cite[Proposition~8.1.2]{BjBr}.

Further, the \emph{major index} and \emph{descent} statistics are
\begin{align*}
\maj(w^{\epsilon}) := \sum_{i\in \Des(w^{\epsilon})} i\quad \text{ and }\quad \des(w^{\epsilon}) := |\Des(w^{\epsilon})|.
\end{align*}
For instance, for $\bar{2}\bar{2}12\bar{1}\in B_{(2,3)}$ we have $\Des(\bar{2}\bar{2}12\bar{1})=\Des(\bar{3}\bar{4}15\bar{2})=\{0,1,4\}$, hence $\maj(\bar{2}\bar{2}12\bar{1})=5$ and $\des(\bar{2}\bar{2}12\bar{1})=~3$.\\
To simplify notation we omit $\epsilon$ and write $w\in B_{\eta}$ instead of $w^{\epsilon}\in B_{\eta}$.

\begin{dfn}\label{dfn:descentPolynomialB_eta}
We denote by
\begin{align*}
d_{B_{\eta}}(t) := \sum_{w \in B_{\eta}} t^{\des(w^{\epsilon})} \in \mathbb{Z}[t]
\end{align*}
the \emph{descent polynomial of $B_{\eta}$}.
We call the coefficients of the descent polynomial \emph{generalised Eulerian numbers of type~$\mathsf{B}$}.
\end{dfn}
Note that for $\eta = (1,\dots,1)$, we have $B_{\eta} = B_n$, so the coefficients of $d_{B_{n}}$ are the Eulerian numbers of type~$\mathsf{B}$. 
Analogously to the definition of Carlitz's $q$-Eulerian polynomial for $S_{\eta}$ we define a type~$\mathsf{B}$-analogue:
\begin{dfn}\label{dfn:Carlitz-q-Eulerian-polynomialB_eta}
The bivariate generating polynomial for the major index and descent statistic over the set of signed multiset permutations is denoted by
\begin{align*}
C_{B_{\eta}}(q,t) := \sum_{w\in B_{\eta}} q^{\maj(w)} t^{\des(w)} \in \mathbb{Z}[q,t].
\end{align*}
\end{dfn} 
For instance, for $\eta = (1,2)$ the generating polynomial of the joint distribution of major index and descent statistic over $B_{(1,2)}$ is given by
$$C_{B_{(1,2)}}(q,t)=  q^3 t^3 + (3 q^3 + 5 q^2 + 3 q) t^2 + (3 q^2 + 5 q + 3) t + 1.$$
Different definitions of major index and descent for signed multiset permutations appear in \cite{FoataHanIII} and \cite{Lin}.
The smallest composition $\eta$ where even our descent polynomial differs from the ones defined in \cite{FoataHanIII} and \cite{Lin} is $\eta=(1,2)$:
for $\sum_{w \in B_{\eta}} t^{\des(w)}$ we obtain
\begin{align*}
2 t^4 + 7 t^3 + 9 t^2 + 5 t + 1 \quad &\text{ using the definition of the descent statistic in \cite{FoataHanIII}},\\
9 t^2 + 14 t + 1 \quad &\text{ using the one in \cite{Lin}},\\
\text{and }\qquad t^3 + 11 t^2 + 11 t + 1 \quad &\text{ using Definition~\ref{dfn:descent_set-B_eta}}.
\end{align*}

Our goal is to construct for each $\eta$ and $X\in \{S_{\eta},B_{\eta}\}$ a polytope such that the generating function 
\begin{align*}
\frac{C_X(q,t)}{\prod_{i=0}^n (1-q^it)}
\end{align*}
is a $q$-analogue of its Ehrhart series. For $X=S_{\eta}$ this is \eqref{MacMahonFormula}, which is MacMahon's formula of type $\mathsf{A}$ (see Theorem~\ref{MacMahonA} for the details).

\subsection{Ehrhart theory}\label{subsection:EhrhartTheory}

As we now explain, both rational functions in \eqref{MacMahonFormula} and in \eqref{MacMahonFormulaTypeB} may be interpreted as $q$-analogues of Ehrhart series of certain polytopes. We start with the special case where $q=1$, viz.\ classical Ehrhart theory.

\subsubsection{Classical Ehrhart theory}

Throughout, let $\mathcal{P}=\mathcal{P}_n$ be an $n$-dimensional lattice polytope in $\mathbb{R}^n$. 
The \emph{lattice point enumerator} of $\mathcal{P}$ is the function $\Lp_{\mathcal{P}}\colon \mathbb{N}_0 \to \mathbb{N}_0$ given by $\Lp_{\mathcal{P}}(k):=|k\mathcal{P} \cap \mathbb{Z}^n|$. 
For details on polytopes and Ehrhart theory see \cite{BeckRobins} and \cite{Ziegler}.\\
A fundamental result in this theory is Ehrhart's Theorem~\cite{Ehrhart}, which states that the function $\Lp_{\mathcal{P}}(k)$ is a polynomial in $k$, the so-called \emph{Ehrhart polynomial}. Equivalently, its generating function, the \emph{Ehrhart series} of $\mathcal{P}$, is of the form
\begin{align*}
\Ehr_{\mathcal{P}}(t) := \sum_{k \geq 0} \Lp_{\mathcal{P}}(k) t^k = \frac{h_{\mathcal{P}}^*(t)}{(1-t)^{n+1}} \in \mathbb{Q}(t),
\end{align*}
where the numerator, the so-called \emph{$h^*$-polynomial} of $\mathcal{P}$, has degree at most $n$.

Lattice point enumeration behaves well under taking products:
Ehrhart series of products of polytopes can be described in terms of \emph{Hadamard products}.
For series $A(t)=\sum_{k\geq 0} a_k t^k, B(t)=\sum_{k\geq 0} b_k t^k \in \mathbb{Q}(t)$ we denote their Hadamard product (with respect to~$t$) by $(A\ast B)(t) := \sum_{k\geq 0} a_k b_k t^k$. 
\begin{rem}
For a composition $\eta=(\eta_1,\dots,\eta_r)$ of $n$, let $\mathcal{P}_{\eta_i}$ be an $\eta_i$-dimensional polytope for $i\in [r]$.
Further, for $i\in [r]$ let $\Lp_{\mathcal{P}_{\eta_i}}$ be the Ehrhart polynomial and $\Ehr_{\mathcal{P}_{\eta_i}}$ the Ehrhart series of $\mathcal{P}_{\eta_i}$. 
The product $\mathcal{P}_{\eta} := \mathcal{P}_{\eta_1}\times \dots \times \mathcal{P}_{\eta_r}$ is an $n$-dimensional polytope with Ehrhart polynomial $\prod_{i=1}^r \Lp_{\mathcal{P}_{\eta_i}}$.
Therefore, its Ehrhart series is given by
\begin{align*} 
\Ehr_{\mathcal{P}_{\eta}}(t) = \sum_{k\geq 0} \left( \prod_{i=1}^r \Lp_{\mathcal{P}_{\eta_i}}(k)\right) t^k = \Conv_{i=1}^r \Ehr_{\mathcal{P}_{\eta_i}}(t).
\end{align*}
\end{rem}
The polytopes which are relevant for us are products of simplices or cross polytopes. 
It turns out that they provide a connection to permutation statistics, which we specify in Theorems~\ref{MacMahonA} and~\ref{MacMahonB}.
\begin{dfn}\label{dfn_simplex_cp}
The \emph{$n$-dimensional standard simplex} is the convex hull of zero and the unit vectors, denoted by 
$$\Delta_n := \conv\{0,e_1,\dots ,e_n\} = \{x\in \mathbb{R}^n : 0\leq x_1+ \dots + x_n \leq 1\}.$$ 
The \emph{$n$-dimensional cross polytope} is the convex hull of the unit vectors and their negatives:
$$\cp_n := \conv\{e_1,-e_1,\dots ,e_n,-e_n\} = \{x\in \mathbb{R}^n : |x_1|+ \dots + |x_n| \leq 1\}.$$
\end{dfn}
\begin{exa}\label{ExampleEhrhart_polynomials_series}
The $h^*$-polynomials of products of standard simplices and cross polytopes can be described through permutations statistics:
\begin{itemize}
\item[(a)] The Ehrhart series of the $n$-dimensional standard simplex $\Delta_n$ is given by $$\Ehr_{\Delta_n}(t)= \sum_{k\geq 0} \binom{n+k}{n} t^k = \frac{1}{(1-t)^{n+1}} = \frac{d_{S_{(n)}}(t)}{(1-t)^{n+1}}.$$
For the $n$-dimensional unit cube $\Box_n := [0,1]^n$, which is the product of $n$ one-dimensional simplices, we obtain $$\Ehr_{\Box_n}(t) = \sum_{k\geq 0} (k+1)^n t^k =\frac{d_{S_n}(t)}{(1-t)^{n+1}}.$$
\item[(b)] The Ehrhart series of the $n$-dimensional cross polytope $\cp_n$ is given by $$\Ehr_{\cp_n}(t)= \sum_{k\geq 0} \sum_{j=0}^n \binom{n}{j}\binom{k+n-j}{n} t^k = \frac{(1+t)^n}{(1-t)^{n+1}} = \frac{d_{B_{(n)}}(t)}{(1-t)^{n+1}}.$$
For the $n$-dimensional cube (centred at the origin) 
$\mboxdot_n := [-1,1]^n$, which is the product of $n$ one-dimensional cross polytopes, we obtain $$\Ehr_{\mboxdot_n}(t)= \sum_{k\geq 0} (2k+1)^n t^k = \frac{d_{B_n}(t)}{(1-t)^{n+1}}.$$
\end{itemize}
\end{exa}
The first three Ehrhart series can be found in \cite[Section 2]{BeckRobins}, the last one follows from Corollary~\ref{Cor_hypercubein0}, a special case of Theorem~\ref{MacMahonB}, which was already shown in \cite[Theorem~3.4]{Brenti}.
The description of the $h^*$-polynomials in terms of descent polynomials over a suitable set of permutations follows from Theorems~\ref{MacMahonA} and~\ref{MacMahonB}.

\subsubsection{Reciprocity results}
We recall reciprocity results in the context of Ehrhart theory and study functional equations of Ehrhart series of products of polytopes.\\
Let $\Lp_{\mathcal{P}^{\circ}}(k)$ denote the numbers of lattice points in the relative interior of the $k$th dilate of~$\mathcal{P}$.
E.g.~for the two-dimensional cross polytope $\cp_2$ and its second dilate the numbers of lattice points in its interior is $\Lp_{\cp_2^{\circ}}(1)=1$, respectively $\Lp_{\cp_2^{\circ}}(2)=5$.

\begin{thm}[Ehrhart--Macdonald Reciprocity, {\cite[Theorem~4.1]{BeckRobins}}]\label{Ehrhart-Macdonald-reciprocity}
For a polytope $\mathcal{P}$ the Ehrhart polynomial satisfies the reciprocity law
\begin{align*}
\Lp_{\mathcal{P}}(-k) = (-1)^n \Lp_{\mathcal{P}^{\circ}}(k).
\end{align*}
\end{thm}

\noindent Similarly to the definition of the Ehrhart polynomial for the interior of a polytope $\mathcal{P}$, we define the Ehrhart series for $\mathcal{P}^{\circ}$ as
\begin{align*}
\Ehr_{\mathcal{P}^{\circ}}(t) := \sum_{k \geq 1} \Lp_{\mathcal{P}^{\circ}}(k) t^k.
\end{align*}
Therefore, Theorem~\ref{Ehrhart-Macdonald-reciprocity} is equivalent to a reciprocity of the generating function
\begin{align*}
\Ehr_{\mathcal{P}}(t^{-1}) = (-1)^{n+1} \Ehr_{\mathcal{P}^{\circ}}(t).
\end{align*}
 
A polytope $\mathcal{P}$ is \emph{reflexive} if, after a suitable translation, $\mathcal{P}=\{x\in\mathbb{R}^n:Ax\leq\textbf{1}\}$ for some $A\in \Mat_n(\mathbb{Z})$ and $0\in\mathcal{P}^{\circ}$.
Equivalently, $\Lp_{\mathcal{P}^{\circ}}(k) = \Lp_{\mathcal{P}}(k-1)$ holds for all $k\in \mathbb{N}_0$.\\
An $n$-dimensional lattice polytope $\mathcal{P}$ is \emph{Gorenstein (of index $l$)} (or equivalently \emph{of codegree~$l$}) if there exists an $l\in \mathbb{N}$ such that
\begin{align*}
& \Lp_{\mathcal{P}^{\circ}}(l-1) = 0,\ \Lp_{\mathcal{P}^{\circ}}(l) = 1 \\
\text{and } & \Lp_{\mathcal{P}^{\circ}}(k) = \Lp_{\mathcal{P}}(k-l) \quad \forall\ k>l.
\end{align*}

\begin{rem}\ \label{Remark:Reflexive_Gorenstein} 
\begin{enumerate} 
\item[(i)] A lattice polytope is reflexive if and only if it is Gorenstein of index $1$.
\item[(ii)] A lattice polytope is Gorenstein of index $l$ if and only if $l\mathcal{P}$ is reflexive, i.e.\
	\begin{align*}
		l \mathcal{P}= \{ x\in \mathbb{R}^n : Ax\leq \textbf{1} \}
	\end{align*}
	for some $A\in \Mat_n(\mathbb{Z})$ and $\textbf{1}$ the all-one vector.
\end{enumerate}
\end{rem}

Whether a polytope is reflexive or not can also be seen from its Ehrhart series:
a polytopes is reflexive if and only if the Ehrhart--Macdonald Reciprocity turns into a self-reciprocity, which is also known as Hibi's palindromic theorem (see, e.g.,~\cite[Theorem~4.6]{BeckRobins}).
The next proposition is an extension of this result to Gorenstein polytopes; see~\cite[Exercise 4.8]{BeckRobins}.
\begin{prop}\label{thm_Gorenstein_functional_eq}
An $n$-dimensional lattice polytope $\mathcal{P}$ is Gorenstein of index $l$ if and only if 
\begin{align*}
\Ehr_{\mathcal{P}}(t^{-1}) = (-1)^{n+1} t^l \Ehr_{\mathcal{P}}(t),
\end{align*}
i.e.\ the nonzero coefficients of the $h^*$-polynomial are symmetric. 
\end{prop}

The property of being Gorenstein behaves well under taking products: 
\begin{prop} \label{prop_Gorensteinproduct}
Let $\eta=(\eta_1,\dots ,\eta_r)$ be a composition of $n$ and $\mathcal{P}_{\eta_i}$ be an $\eta_i$-dimensional lattice polytopes for $i\in [r]$. 
The product $\mathcal{P}_{\eta} = \prod_{i=1}^r \mathcal{P}_{\eta_i}$ is Gorenstein of index $l$ if and only if every $\mathcal{P}_{\eta_i}$ is Gorenstein of index $l$.
\end{prop}
\begin{proof}
First of all, we observe that the index of a Gorenstein polytope is uniquely determined.
Now assume $\mathcal{P}_{\eta}$ is Gorenstein of index $l$, i.e.\ after a suitable translation 
	\begin{align*}
		l \mathcal{P}_{\eta} = \prod_{i=1}^r l \mathcal{P}_{\eta_i} = \{ x\in \mathbb{R}^n : Ax \leq \textbf{1} \}.
	\end{align*}
The matrix $A$ can be chosen as a block matrix with matrices $A_i$ on the diagonal such that after a suitable translation $l \mathcal{P}_{\eta_i} = \{ x\in \mathbb{R}^{\eta_i} : A_i x\leq \textbf{1}\}$. Thus $l \mathcal{P}_{\eta_i}$ is reflexive and therefore $\mathcal{P}_{\eta_i}$ is Gorenstein of index $l$ for all $i\in [r]$.\\
On the other hand, if all $\mathcal{P}_{\eta_i}$ are Gorenstein of index $l$, i.e.\ the $l$th dilate is of the form
	\begin{align*}
		l \mathcal{P}_{\eta_i} = \{ x\in \mathbb{R}^{\eta_i} : A_i x \leq \textbf{1} \},
	\end{align*}
it follows that 
	\begin{align*}
		l \mathcal{P}_{\eta} = \prod_{i=1}^r l \mathcal{P}_{\eta_i} = \{ x\in \mathbb{R}^n : Ax \leq \textbf{1} \}
	\end{align*}
for a matrix $A$ as described above.
\end{proof}

By Remark~\ref{Remark:Reflexive_Gorenstein}~(i) we obtain an analogue of Proposition~\ref{prop_Gorensteinproduct} for reflexive polytopes.
\begin{cor}
The product $\mathcal{P}_{\eta} = \prod_{i=1}^r \mathcal{P}_{\eta_i}$ of $\eta_i$-dimensional polytopes $\mathcal{P}_{\eta_i}$ is reflexive if and only if every $\mathcal{P}_{\eta_i}$ is reflexive.
\end{cor}
The next corollary follows immediately from Propositions~\ref{thm_Gorenstein_functional_eq} and~\ref{prop_Gorensteinproduct}.
\begin{cor} \label{cor_functionalequation_HP}
For polytopes $\mathcal{P}_{\eta}$ and $\mathcal{P}_{\eta_i}$, $i \in [r]$, as in Proposition~\ref{prop_Gorensteinproduct}
the Ehrhart series of $\mathcal{P}_{\eta}$ satisfies a functional equation of the form 
	\begin{align*}
		\Ehr_{\mathcal{P}_{\eta}}(t^{-1}) = (-1)^{n+1} t^{l} \Ehr_{\mathcal{P}_{\eta}}(t)
	\end{align*}
if and only if each of the Hadamard factors satisfies a functional equation, i.e.\
	\begin{align*}
		\Ehr_{\mathcal{P}_{\eta_i}}(t^{-1}) = (-1)^{\eta_i+1} t^{l} \Ehr_{\mathcal{P}_{\eta_i}}(t).
	\end{align*}
\end{cor}

\begin{rem}
The equivalence in Corollary~\ref{cor_functionalequation_HP} is remarkable in the sense that in general for  generating functions of polynomials only one direction, namely the reverse one, follows.
\end{rem}
We close the section with the following example which yields palindromicity statements of the generalised Eulerian numbers of types~$\mathsf{A}$ and~$\mathsf{B}$ in Remark~\ref{rem_palindromicity}.
\begin{exa}\label{exa_Gorenstein}\
\begin{enumerate}
\item[(i)] The $n$-dimensional standard simplex is Gorenstein of index $n+1$.
\item[(ii)] The product of standard simplices $\Delta_{\eta} = \prod_{i=1}^r \Delta_{\eta_i}$ is Gorenstein (of index $l$) if and only if all factors have the same dimension, i.e.\ $\eta_i = \eta_j (=l-1)$ for all $i,j\in [r]$.
\item[(iii)] The $n$-dimensional cross polytope is reflexive.
\item[(iv)] The product of cross polytopes $\cp_{\eta} = \prod_{i=1}^r \cp_{\eta_i}$ is reflexive for every $\eta$.
\end{enumerate}
\end{exa}

In the next section we refine our description of the $h^*$-polynomial of (products of) standard simplices and cross polytopes by introducing weight functions on $\mathbb{Z}^n$ or, equivalently in the language of permutation statistics, by taking the major index into account.

\subsubsection{Weighted Ehrhart theory}

For a variable $q$ we consider $q$-analogues of Ehrhart series of standard simplices and cross polytopes by refining the lattice point enumeration. 
That is, for a suitable polytope $\mathcal{P}$ we define a family $\boldsymbol{\mu}_n= \left( \mu_{k,n} \right)_{k\in \mathbb{N}_0}$ of functions $\mu_{k,n} \colon k\mathcal{P}\cap \mathbb{Z}^n \to \mathbb{N}_0$ such that the refinement on the side of the Ehrhart series corresponds to the one by the major index on the permutation side, more precisely
\begin{equation}\label{eq:aim_weightedES}
\Ehr_{\mathcal{P},\boldsymbol{\mu}_n}(q,t) := \sum_{k\geq 0} \sum_{x\in k\mathcal{P}\cap \mathbb{Z}^n} q^{\mu_{k,n}(x)} t^k = \frac{C_X(q,t)}{\prod_{i=0}^n (1-q^i t)}
\end{equation}
for pairs $(\mathcal{P},X)$ as in Example~\ref{ExampleEhrhart_polynomials_series}, viz.\ $(\Delta_n,S_{(n)})$, $(\Box_n,S_{n})$, $(\cp_n,B_{(n)})$, $(\mboxdot_n,B_{n})$, and more generally $(\Delta_{\eta},S_{\eta})$ and $(\cp_{\eta},B_{\eta})$.
Since the refinement on the lattice point enumeration results from weighting the lattice points, we call the functions $\mu_{k,n}$ weight functions.
The functions we define are inspired by~\cite{Chapoton}, partially building on earlier work of~\cite{Stanley72}.
Another different refinement of the Ehrhart series of the simplex, the hypercube and the cross polytope using commutative algebra is developed in~\cite{AdeyemoSzendroi}.
First, we discuss the weight functions defined by Chapoton and explain why new ones were needed for our purpose.

Under certain assumptions, in \cite{Chapoton} a $q$-analogue of the Ehrhart series is defined by introducing a weight function $\lambda_n$ which is also a linear form on $\mathbb{R}^n$:

\begin{ass}\label{Conditionslinearform}
Assume that the pair $(\mathcal{P}, \lambda_n)$ satisfies 
\begin{align*} 
\begin{split}
\lambda_n &\geq 0 \text{ for all } x\in \mathcal{P} \text{ and }\\
\lambda_n(x) &\neq \lambda_n(y) \text{ for every edge }x-y \text{ of } \mathcal{P}.
\end{split}
\end{align*}
\end{ass}

\begin{thm}[{\cite[Proposition~2.1]{Chapoton}}]\label{weighted_Ehrhart_series}
For a linear weight function $\lambda_n$ and a polytope $\mathcal{P}$ satisfying Assumption~\ref{Conditionslinearform} the generating function of $\sum_{x\in k\mathcal{P}} q^{\lambda_n(x)}$, the so-called \emph{$q$-Ehrhart series}, 
\begin{align*}
\Ehr_{\mathcal{P},\lambda_n}(q,t) := \sum_{k\geq 0} \sum_{x\in k\mathcal{P}\cap \mathbb{Z}^n} q^{\lambda_n(x)} t^k
\end{align*}
is a rational function in $q$ and $t$.
\end{thm}
The weighted Ehrhart series as in Theorem~\ref{weighted_Ehrhart_series} are special cases of integer point transform, see, e.g., \cite[Chapter~3]{BeckRobins}.
Chapoton \cite{Chapoton} studied a special case in which there exists a $q$-analogue of the Ehrhart polynomial.
Indeed, Theorem~\ref{weighted_Ehrhart_series} is equivalent to showing that the weighted sum $\sum_{x\in k\mathcal{P}} q^{\lambda_n(x)}$ is given by a polynomial $\Lp_{\mathcal{P}, \lambda_n}(z)\in \mathbb{Q}(q)[z]$. We refer to the proof of \cite[Theorem~3.1]{Chapoton} for the details.
\begin{thm}[{\cite[Theorem~3.1]{Chapoton}}]\label{q-Ehrhart_polynomials}
For a linear weight function $\lambda_n$ and a polytope $\mathcal{P}$ satisfying Assumption~\ref{Conditionslinearform} there exists a polynomial $\Lp_{\mathcal{P}, \lambda_n}(z)\in \mathbb{Q}(q)[z]$ of degree $\max\{\lambda_n(x) : x\in\mathcal{P}\}$ such that for all $k\in\mathbb{N}_0$
\begin{align*}
\Lp_{\mathcal{P},\lambda_n}([k]_q) = \sum_{x\in k\mathcal{P}\cap \mathbb{Z}^n} q^{\lambda_n(x)}.
\end{align*}
For a fixed weight $\lambda_n$ we call $\Lp_{\mathcal{P},\lambda_n}$ the \emph{$q$-Ehrhart polynomial} of $\mathcal{P}$.
\end{thm}
Note that for $q=1$ we obtain the classical Ehrhart series and polynomial. 
In Theorems~\ref{MacMahonA} and~\ref{MacMahonB} we give further examples of $q$-Ehrhart polynomials in a different setting than Assumption~\ref{Conditionslinearform}.\\
For instance, Assumption~\ref{Conditionslinearform} is satisfied for $\lambda_n$ defined by $\lambda_n(x)=\sum_{i=1}^n x_i$, the unit cube and a simplex $\mathcal{O}(\Delta_n):= \conv \{ 0, e_n, e_{n-1}+e_n, \dots ,\sum_{i=1}^n e_i \}$ which is similar to the standard simplex and which we specify later.
The next example can be found in \cite[Section~5]{Chapoton} or rather \cite[\S 8]{Stanley72}.
\begin{exa}\label{exa_q-Ehrhartseries_Chapoton}
Let $\lambda_n$ be the linear form on $\mathbb{R}^n$ given by $\lambda_n(x)=\sum_{i=1}^n x_i$ for $x\in\mathbb{R}^n$.
\begin{itemize}
\item[(i)] The $q$-Ehrhart series of the $n$-dimensional simplex $\mathcal{O}(\Delta_n)$ is given by 
\begin{align*}
\Ehr_{\mathcal{O}(\Delta_n),\lambda_n}(q,t)= \sum_{k\geq 0} \binom{k+n}{n}_q t^k = \frac{1}{\prod_{i=0}^n (1-q^i t)}.
\end{align*}
\item[(ii)] The $q$-Ehrhart series of the $n$-dimensional unit cube $\Box_n$ is 
\begin{align*}
\Ehr_{\Box_n,\lambda_n}(q,t) = \sum_{k\geq 0} [k+1]_q^n t^k =\frac{C_{S_n}(q,t)}{\prod_{i=0}^n (1-q^i t)}.
\end{align*}
\end{itemize}
\end{exa}
\noindent Note that \eqref{eq:aim_weightedES} is fulfilled for the hypercube and $\lambda_n$ as defined above. 

But Assumption~\ref{Conditionslinearform} is restrictive in the following sense: 
considering the four polytopes in Example~\ref{ExampleEhrhart_polynomials_series} and $\lambda_n$ defined as the sum of the coordinates, Assumption~\ref{Conditionslinearform} is only satisfied for the unit cube. 
In particular, we can not compute the $q$-analogue of the Ehrhart series (in the sense of Theorem~\ref{weighted_Ehrhart_series}) of the standard simplex $\Delta_n$ and the weight function $\lambda_n$. 
Considering the $q$-Ehrhart series of $\mathcal{O}(\Delta_n)$ in Example~\ref{exa_q-Ehrhartseries_Chapoton}~(i), we notice that the numerator is equal to $C_{S_{(n)}}(q,t)=1$ and the polytope $\mathcal{O}(\Delta_n)$ itself is not too far away from the standard simplex.
More precisely, given a chain $1<\dots <n$ the standard $n$-simplex $\Delta_n$ is called the chain polytope and $\mathcal{O}(\Delta_n)$ is called the order polytope of this chain. The chain and order polytope share a number of properties, for example, their Ehrhart series coincide; cf.~\cite[Theorem~4.1]{Stanley2p1p}.

Next, we define a bijection between the chain polytope and the order polytope of the $n$-chain. 
As we will see later, this behaves well (i.e.\ the pair $(\Delta_n, S_{(n)})$ fulfils \eqref{eq:aim_weightedES}) when we equip $\mathcal{O}(\Delta_n)$ with the linear form defined by the sum of the coordinates.  

\noindent We omit the routine computations of the proof of the following lemma:
\begin{lem}\label{lem:bij_orderpolytope}
For $k\in \mathbb{N}$, the map
\begin{align*}
\phi_{k,n} \colon k\Delta_n &\to k \mathcal{O}(\Delta_n) \\
(x_1,\dots ,x_n) &\mapsto (k-\sum_{i=1}^n x_i, k-\sum_{i=2}^n x_i, \dots , k-x_n)
\end{align*}
is a bijection between the $k$th dilate of the chain polytope and the $k$th dilate of the order polytope of the $n$-chain.
\end{lem}

\begin{figure}
\begin{subfigure}[t]{0.33\textwidth}
\includegraphics[width=\textwidth]{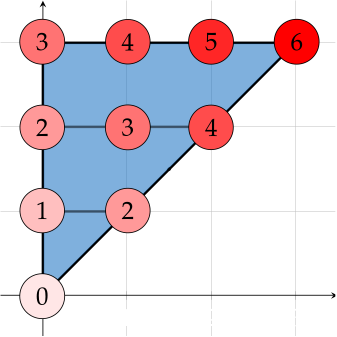}
\subcaption{$\lambda_2$ on  $3\mathcal{O}(\Delta_2)$}
\end{subfigure}
   \hspace{.045\linewidth}
\begin{subfigure}[t]{0.33\textwidth}
\includegraphics[width=\textwidth]{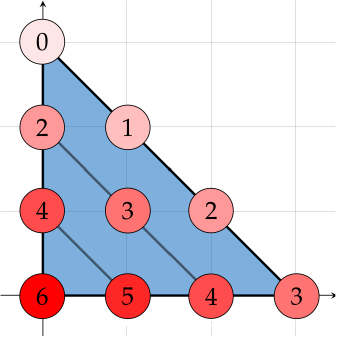}
\subcaption{$\lambda_2 \circ \phi_{3,2}$ on $3\Delta_2$}
\end{subfigure}
\caption{Weight functions on $3\mathcal{O}(\Delta_2)$ and $3\Delta_2$, respectively.}
\label{Fig:Bijection_weights}
\end{figure}
Using the bijection $\phi_{k,n}$ we define a weight function on the $k$th dilate of the standard simplex which is illustrated by Figure~\ref{Fig:Bijection_weights} for $k=3$ and $n=2$: by applying $\phi_{3,2}$, every lattice point in $3\Delta_2$ corresponds to point in $y\in 3\mathcal{O}(\Delta_2)$ which is equipped with the weight $\lambda_2(y)$. Thus we obtain the weight $(\lambda_2\circ \phi_{3,2})(x)$ of $x\in 3\Delta_2$.\\
For general $k$ and $n$ we determine the following:

\begin{dfn}\label{dfn:weightfunction-simplex}
For $k\in \mathbb{N}_0$ and $n\in \mathbb{N}$ we define the weight functions
\begin{align*}
\mu_{k,n}\colon k\Delta_n &\to \mathbb{N}_0 \\
x &\mapsto ( \lambda_n \circ \phi_{k,n} ) (x)
\end{align*}  
on the $n$-dimensional standard simplex. We write $\boldsymbol{\mu}_n$ for the family of weight functions $\left(\mu_{k,n}\right)_{k\in \mathbb{N}_0}$.
\end{dfn}

Analogously to Theorem~\ref{weighted_Ehrhart_series}, we define the \emph{weighted Ehrhart series of the standard simplex} to be
\begin{align*}
\Ehr_{\Delta_n,\boldsymbol{\mu}_n}(q,t) := \sum_{k\geq 0} \sum_{x\in k\Delta_n\cap \mathbb{Z}^n} q^{\mu_{k,n}(x)} t^k. 
\end{align*}
Note that for $q=1$ we obtain the classical Ehrhart series.

Next, we extend the weight functions on standard simplices to weight functions on products of those.
For products of the simplices $\Delta_{\eta_1}, \dots , \Delta_{\eta_r}$ the map
\begin{align*}
\phi_{k,\eta} := \prod_{i=1}^r \phi_{k,\eta_i} \colon \prod_{i=1}^r k\Delta_{\eta_i} &\to \prod_{i=1}^r \mathcal{O}(k\Delta_{\eta_i}), \\
 (x_1,\dots , x_r) &\mapsto (\phi_{k,\eta_1}(x_1), \dots ,\phi_{k,\eta_r}(x_r))
\end{align*}
(induced by the bijection in Lemma~\ref{lem:bij_orderpolytope}) is a bijection between the product of the chain polytopes and the product of the order polytopes of standard $\eta_i$-simplices.

\begin{dfn}\label{dfn:weight_simplices}
For $k\in \mathbb{N}_0$ and $\eta$ a composition of $n$, we define weight functions on the product of simplices $\Delta_{\eta}=\prod_{i=1}^r \Delta_{\eta_i}$ 
\begin{align*}
 \prod_{i=1}^r \mu_{k,\eta_i} \colon k\Delta_{\eta} & \to \mathbb{N}_0, \\
x & \mapsto (\lambda_n \circ \phi_{k,\eta})(x).
\end{align*}
More precisely, an element $x=(x_1,\dots ,x_r)\in k\Delta_{\eta}$ is sent to $$\lambda_n(\phi_{k,\eta_1}(x_1),\dots , \phi_{k,\eta_r}(x_r)) = \mu_{k,\eta_1}(x_1)+ \dots + \mu_{k,\eta_r}(x_r).$$
We denote by $\boldsymbol{\mu}_{\eta}$ the family of the weight functions above.
\end{dfn}
The Ehrhart series of products of standard simplices is a Hadamard product with respect to the variable $t$, more precisely
\begin{align} \label{eq:2.2.2nr.1}
\Ehr_{\Delta_{\eta}, \boldsymbol{\mu}_n}(q,t) = \Conv_{i=1}^r \Ehr_{\Delta_{\eta_i},\boldsymbol{\mu}_{\eta_i}} (q,t).
\end{align}
This leads to a $q$-analogue of Example~\ref{ExampleEhrhart_polynomials_series}~(a): 
\begin{exa}\ \label{exa-q-Ehrhartseries}
\begin{itemize}
\item[(i)] The weighted Ehrhart series of the $n$-dimensional standard simplex is 
\begin{align*}
\Ehr_{\Delta_n, \boldsymbol{\mu}_n} (q,t) = \sum_{k\geq 0} \binom{n+k}{n}_q t^k = \frac{1}{\prod_{i=0}^n (1-q^i t)} = \frac{C_{S_{(n)}}(q,t)}{\prod_{i=0}^n (1-q^i t)}.
\end{align*}
\item[(ii)] For $\eta=(1,\dots , 1)$ the weighted Ehrhart series of the $n$-dimensional unit cube is given by
\begin{align*}
\Ehr_{\Box_n, \boldsymbol{\mu}_{\eta}} (q,t) = \sum_{k\geq 0} \binom{1+k}{1}_q^n t^k = \frac{C_{S_n}(q,t)}{\prod_{i=0}^n (1-q^i t)}.
\end{align*}
\end{itemize}
\end{exa}
\begin{proof}\
\begin{itemize}
\item[(i)] For the weighted Ehrhart series of the standard simplex we obtain
\begin{align*}
\Ehr_{\Delta_n,\boldsymbol{\mu}_n}(q,t) =&  \sum_{k\geq 0} \sum_{x\in k\Delta_n\cap \mathbb{Z}^n} q^{\lambda_n(\phi_{k,n}(x))} t^k \\
=& \sum_{k\geq 0} \sum_{x\in k\mathcal{O}(\Delta_n)\cap \mathbb{Z}^n} q^{\lambda_n(x)} t^k \\
=& \frac{1}{\prod_{i=0}^n (1-q^i t)}, 
\end{align*}
where the last equality follows from Example~\ref{exa_q-Ehrhartseries_Chapoton}~(i).
\item[(ii)] For $n=1$ the chain and order polytope of a $1$-chain coincide. Thus
\begin{align*}
\Ehr_{\Box_n, \boldsymbol{\mu}_n} (q,t) = \Ehr_{\Box_n, \lambda_{\eta}} (q,t) = \sum_{k\geq 0} \binom{1+k}{1}_q^n t^k = \frac{\sum_{w\in S_n} q^{\maj(w)} t^{\des(w)}}{\prod_{i=0}^n (1-q^i t)},
\end{align*}
as in Example~\ref{exa_q-Ehrhartseries_Chapoton}~(ii).
\end{itemize} \vspace{-0.5cm}
\end{proof}
\noindent Generalising the last example to products of $\eta_i$-dimensional standard simplices and to permutation statistics on $S_{\eta}$, we obtain MacMahon's formula of type~$\mathsf{A}$ (Theorem~\ref{MacMahonA}).\\

Next, we extend the weight functions above to ones on cross polytopes.
Every lattice point in the $k$th dilate of the cross polytope can be mapped to a unique lattice point in the $k$th dilate of the standard simplex by taking absolute values of the entries.
To each point in there we associated a weight via the weight function $\mu_{k,n}$; see Definition~\ref{dfn:weight_simplices}.
In other words, we obtain weights on lattice points in the cross polytope by extending the one on the standard simplex via reflections along coordinate hyperplanes.
This is illustrated in Figure~\ref{fig:weights-on-3cp_2} for the third dilate of the two-dimensional cross polytope.

\begin{figure}
\begin{subfigure}[t]{0.33\textwidth}
\includegraphics[width=\textwidth]{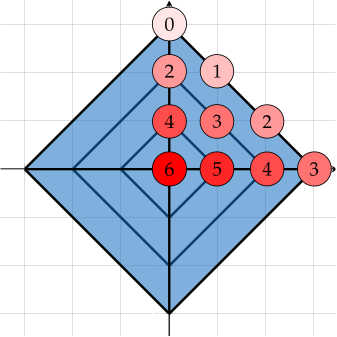}
\subcaption{weights defined by $\mu_{3,2}$ on $3\Delta_2$}
\end{subfigure}
   \hspace{.045\linewidth}
\begin{subfigure}[t]{0.33\textwidth}
\includegraphics[width=\textwidth]{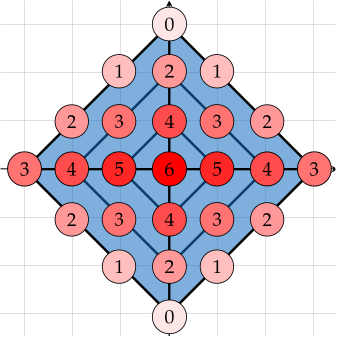}
\subcaption{weights on $3\protect\cp_2$}
\end{subfigure}
     \caption{How to obtain weights on $3\protect\cp_2$ from $\mu_{3,2}$ on $3\Delta_2$.}
     \label{fig:weights-on-3cp_2}
\end{figure}

\begin{dfn}\label{dfn:weightfunction_cp}
For $k\in \mathbb{N}_0$ and $|\cdot |\colon \mathbb{Z}^n \to \mathbb{N}_0^n$ where $x\mapsto |x|:=(|x_1|,\dots ,|x_n|)$ we define weight functions
\begin{align*}
\bar{\mu}_{k,n}\colon k\cp_n & \to \mathbb{N}_0 \\
x & \mapsto \mu_{k,n}(|x|)=(\lambda_n \circ \phi_{k,n})(|x|)
\end{align*}
on the $n$-dimensional cross polytope.
We denote by $\bar{\boldsymbol{\mu}}_n$ the family of functions $\left( \bar{\mu}_{k,n} \right)_{k\in \mathbb{N}_0}$.
\end{dfn}
\noindent Note that $\bar{\boldsymbol{\mu}}_n$ restricted to the standard simplex $\Delta_n$ is equal to $\boldsymbol{\mu}_n$.

This way we define the \emph{weighted Ehrhart series of the cross polytope}
\begin{align*}
\Ehr_{\cp_n,\bar{\boldsymbol{\mu}}_n}(q,t) := \sum_{k\geq 0} \sum_{x\in k\cp_n\cap \mathbb{Z}^n} q^{\bar{\mu}_{k,n}(x)} t^k = \sum_{k\geq 0} \sum_{x\in k\cp_n \cap \mathbb{Z}^n} q^{\lambda_n(\phi_{k,n}(|x|))} t^k.
\end{align*}
Analogously to Definition~\ref{dfn:weight_simplices} we define weight functions on products of cross polytopes:
\begin{dfn}\label{dfn:weight_cp}
For $k\in \mathbb{N}_0$ and $\eta$ a composition of $n$, we define weight functions on the product of cross polytopes $\cp{\eta}=\prod_{i=1}^r \cp{\eta_i}$ to be
\begin{align*}
\prod_{i=1}^r \bar{\mu}_{k,\eta_i} \colon k\cp{\eta} & \to \mathbb{N}_0 \\
x & \mapsto \mu_{k,\eta}(|x|)=(\lambda_n \circ \phi_{k,\eta})(|x|).
\end{align*}
More precisely, an element $x=(x_1,\dots ,x_r)\in k\cp{\eta}$ is sent to $$\lambda_n(\phi_{k,\eta_1}(|x_1|),\dots , \phi_{k,\eta_r}(|x_r|))=\bar{\mu}_{k,\eta_1}(x_1)+\dots +\bar{\mu}_{k,\eta_r}(x_r),$$ where $|x_i|\in k\Delta_{\eta_i}$ for every $i$.
We denote by $\bar{\boldsymbol{\mu}}_{\eta}$ the family of the functions above.
\end{dfn}
Analogously to \eqref{eq:2.2.2nr.1} we obtain
\begin{align}\label{eq:2.2.2nr.2}
\Ehr_{\cp_{\eta},\bar{\boldsymbol{\mu}}_n} (q,t) = \Conv_{i=1}^r \Ehr_{\cp_{\eta_i},\bar{\boldsymbol{\mu}}_{\eta_i}}(q,t),
\end{align}
where the Hadamard product is taken with respect to $t$.
The essence of our type-$\mathsf{B}$ analogue of MacMahon's formula (see Theorem~\ref{MacMahonB}) is an explicit description of the numerator of \eqref{eq:2.2.2nr.2}.\\

Both families of weight functions $\boldsymbol{\mu}_{\eta}$ and $\bar{\boldsymbol{\mu}}_{\eta}$ are not linear forms, so this gives rise to a different approach defining $q$-analogues of Ehrhart series, or, equivalently, $q$-Ehrhart polynomials, than Chapoton uses.
Nevertheless the weight functions are in some sense natural as they are defined by the obvious subdivision of the cross polytope into standard simplices, a bijection between a chain ($\Delta_n$) and its corresponding order polytope ($\mathcal{O}(\Delta_n)$) and, in the end, a linear form $\lambda_n$ given by the sum of the coordinates. 
Closing the circle, our main result (Theorem~\ref{MacMahonB}) and its corollaries show that the weight functions $\bar{\boldsymbol{\mu}}_{\eta}$ satisfy \eqref{eq:aim_weightedES}.

\section{MacMahon's formula of type~$\mathsf{B}$}\label{section:MacMahonTypeB}

We obtain an interpretation of MacMahon's formula (of type~$\mathsf{A}$) in terms of weighted Ehrhart series  (more precisely, in terms of $q$-Ehrhart polynomials) of products of standard simplices. 
Likewise, we develop a similar type-$\mathsf{B}$ analogue of MacMahon's formula which admits an interpretation as weighted Ehrhart series of products of cross polytopes.
Recall that $\eta=(\eta_1,\dots ,\eta_r)$ is a composition of an integer $n$ into $r$ parts.

\begin{thm}[MacMahon's formula of type~$\mathsf{A}$, {Theorem~\hyperref[ThmA]{A} made precise}] \label{MacMahonA}
The generating polynomial of the joint distribution of major index and descent statistic over the set of multiset permutations is a $q$-analogue of the $h^*$-polynomial of products of standard simplices, i.e.\
\begin{align}\label{eq:MacMahonA}
\frac{C_{S_{\eta}}(q,t)}{\prod_{i=0}^n (1-q^i t)} = \sum_{k\geq 0} \left( \prod_{i=1}^r \binom{k+\eta_i}{\eta_i}_q \right) t^k = \Ehr_{\Delta_{\eta},\boldsymbol{\mu}_{\eta}}(q,t).
\end{align}
\end{thm}
\begin{proof}
The first identity was proven by MacMahon; cf.~\cite[\S 462, Vol.~2, Ch.~IV, Sect.~IX]{MacMahon}.
Using \eqref{eq:2.2.2nr.1}, Definition~\ref{dfn:weight_simplices} of the weight functions $\boldsymbol{\mu}_{\eta}$ and Example~\ref{exa-q-Ehrhartseries}~(i) leads to the second equality.
\end{proof}

For $\eta=(1,\dots,1)$  the identity specialises to the symmetric group on the left hand side of Theorem~\ref{MacMahonA} (see also Example~\ref{exa-q-Ehrhartseries}~(ii)). 
Since $S_n$ is a Coxeter group of type~$\mathsf{A}$, we refer to Theorem~\ref{MacMahonA} as MacMahon's formula of type~$\mathsf{A}$.\\
By passing from permutations to signed permutations we get the hyperoctahedral group~$B_n$, a Coxeter group of type~$\mathsf{B}$, and its generalisation $B_{\eta}$, cf.~Section~\ref{section:subsignedmultisetperm}.
On the polytope side we consider the cross polytope $\cp_n = \conv\{e_1,-e_1,\dots ,e_n,-e_n\}$ as a signed analogue of the standard simplex $\Delta_n=\conv\{0,e_1,\dots,e_n\}$.
This analogy carries over to the equality of rational functions in Theorem~\ref{MacMahonA}, namely the generating polynomial of the joint distribution of major index and descent statistic over the set of signed multiset permutations over a suitable denominator and the weighted Ehrhart series of products of cross polytopes.

\begin{thm}[MacMahon's formula of type~$\mathsf{B}$, {Theorem~\hyperref[ThmB]{B} made precise}] \label{MacMahonB}
The generating polynomial of the joint distribution of major index and descent statistic over the set of signed multiset permutations is a $q$-analogue of the $h^*$-polynomial of products of cross polytopes, i.e.\
\begin{align}\label{eq:MacMahonB}
\frac{C_{B_{\eta}}(q,t)}{\prod_{i=0}^n(1-q^i t)} &= \sum_{k\geq 0} \left( \prod_{i=1}^r \sum_{j= 0}^{\eta_i} \left( q^{\frac{j(j-1)}{2}} \binom{\eta_i}{j}_q \binom{k-j+\eta_i}{\eta_i}_q \right)\right) t^k= \Ehr_{\cp_{\eta}, \bar{\boldsymbol{\mu}}_{\eta}}(q,t).
\end{align}
\end{thm}
Note that the term in the middle of \eqref{eq:MacMahonB} includes a $q$-Ehrhart polynomial. 
This immediately follows from the proof of \cite[Theorem~3.1]{Chapoton}.
Further, notice that Theorem~\ref{MacMahonB} extends Theorem~\ref{MacMahonA}. 
Indeed, retaining only the summand for $j=0$ in the inner sum of the term in the middle of~\eqref{eq:MacMahonB}, we obtain the weighted Ehrhart series of $\Delta_{\eta}$ on the right hand side and $S_{\eta}$ instead of $B_{\eta}$ on the left hand side; cf.~the proof of Theorem~\ref{MacMahonB}.
MacMahon's formula of type~$\mathsf{B}$ yields the following corollaries, which are $q$-analogues of Example~\ref{ExampleEhrhart_polynomials_series}~(b).

\begin{cor}\label{Cor_cp} 
For $\eta=(n)$, Theorem~\ref{MacMahonB} implies 
\begin{align*}
\frac{C_{B_{(n)}}(q,t)}{\prod_{i=0}^n (1-q^i t)} = \frac{\prod_{i=0}^{n-1} (1+q^i t)}{\prod_{i=0}^n (1-q^i t)} = \Ehr_{\cp_n, \bar{\boldsymbol{\mu}}_{(n)}} (q,t)
\end{align*}
for the weighted Ehrhart series of the $n$-dimensional cross polytope.
\end{cor}

\begin{cor}\label{Cor_hypercubein0}
For $\eta =(1,\dots ,1)$ and therefore $B_{\eta}=B_n$, Theorem~\ref{MacMahonB} implies
\begin{align}\label{eq:Ehrcubecentredin0}
\frac{C_{B_n}(q,t)}{\prod_{i=0}^n (1-q^i t)} = \sum_{k\geq 0} \left( [k+1]_q + [k]_q \right)^n t^k = \Ehr_{\mboxdot_n, \bar{\boldsymbol{\mu}}_n}(q,t)
\end{align}
for the weighted Ehrhart series of the $n$-dimensional cube $\mboxdot_n=[-1,1]^n$.
In particular the Eulerian numbers of type~$\mathsf{B}$ appear as the coefficients of the $h^*$-polynomial of the cube centred at the origin.
\end{cor}
\noindent The first identity of \eqref{eq:Ehrcubecentredin0} is also known by \cite[Equation 26]{ChowGessel}.\\

The key to prove the first identity in Theorem~\ref{MacMahonB} are barred permutations which first appear in a proof of Gessel and Stanley \cite[Section~2]{GesselStanley} and which we briefly recall in the following.\\
Throughout, let $w= w_1\dots w_n \in B_{\eta}$ denote a signed multiset permutation, which means that $w_i\in\{\pm i : i\in [r]\}$. 
We call the space between $w_i$ and $w_{i+1}$ the $i$th space of $w$ for $i\in [n-1]$. 
The space before $w_1$ is called the $0$th space and the one after $w_n$ is the $n$th space of $w$. 
A \emph{barred permutation on $w\in B_{\eta}$} is obtained by inserting bars in those spaces following the rule:
if $i\in \Des(w)$ then there is at least one bar in the $i$th space.
For example, $||12|\bar{1}$ is a barred permutation on $12\bar{1}\in B_{(2,1)}$.\\
Further we define $\mathbb{B}_{\eta}$ to be the set of all barred permutations on elements in $B_{\eta}$, $\mathbb{B}_{\eta}(k)$ the barred permutations in $\mathbb{B}_{\eta}$ with $k$ bars and $\mathbb{B}_{\eta}(w)$ the barred permutations on $w\in B_{\eta}$.
In the example above $||12|\bar{1} \in B_{(2,1)}, \mathbb{B}_{(2,1)}(3), B_{(2,1)}(12\bar{1})$.
Clearly 
\begin{align}\label{barredpermUnion}
\mathbb{B}_{\eta} = \bigcup_{k\geq 0} \mathbb{B}_{\eta}(k) = \bigcup_{w\in B_{\eta}} \mathbb{B}_{\eta}(w),
\end{align}
where the unions are disjoint.\\
We will often refer to the $i$th space of a barred permutation $\beta\in \mathbb{B}_{\eta}(w)$ by which we mean the $i$th space of the permutation $w\in B_{\eta}$.
Further for some barred permutation $\beta \in \mathbb{B}_{\eta}$ we denote by $b_i(\beta)$, $i\in [n]_0$, the number of bars in the $i$th space.
We define the \emph{weight of a barred permutation} 
\begin{align*}
\wt\colon \mathbb{B}_{\eta} & \to \mathbb{Z}[q,t] \\
\beta & \mapsto \wt(\beta)= q^{\sum_{i=0}^n i\, b_i(\beta)} t^{\sum_{i=0}^n b_i(\beta)}.
\end{align*}
For instance, for $\beta=||12|\bar{1}$ we obtain $b_0(\beta)=2, b_1(\beta)=b_3(\beta)=0, b_2(\beta)=1$ and therefore $\wt(\beta)= q^2 t^3$.

\begin{proof}[Proof of Theorem~\ref{MacMahonB}]
The proof is divided into two parts, one for each of the two asserted equalities.

\textbf{Part 1:}
To prove the first identity
\begin{equation}\label{eq:first_identity}
\frac{\sum_{w\in B_{\eta}}q^{\maj(w)}t^{\des(w)}}{\prod_{i=0}^n(1-q^i t)} = \sum_{k\geq 0} \left( \prod_{i=1}^r \sum_{j= 0}^{\eta_i} \left( q^{\frac{j(j-1)}{2}} \binom{\eta_i}{j}_q \binom{k-j+\eta_i}{\eta_i}_q \right)\right) t^k
\end{equation}
we proceed in two steps. 
First, we count all weights of barred permutations on a fixed $w\in B_{\eta}$, which sum up to the term on the left hand side of \eqref{eq:first_identity}. 
Afterwards, we sum over all weights of all barred permutations with a fixed number of bars, which gives the right hand side of \eqref{eq:first_identity}.\\
Let $\beta \in \mathbb{B}_{\eta}(w)$ for some $w\in B_{\eta}$. 
There exists a unique barred permutation $\tilde{\beta}$ on $w$ which has exactly one bar in the $i$th space if $i\in \Des(w)$ and none otherwise. 
So $\tilde{\beta}$ is minimal in $\mathbb{B}_{\eta}(w)$ with respect to its number of bars and therefore called the \emph{minimal barred permutation}.
Clearly $\wt(\tilde{\beta})= q^{\maj(w)} t^{\des(w)}$. 
We obtain all barred permutations in $\mathbb{B}_{\eta}(w)$ by inserting bars in all the spaces of $\tilde{\beta}$. 
Thus
\begin{align*}
\sum_{\beta \in \mathbb{B}_{\eta}(w)} \wt(\beta) =&\ q^{\maj(w)} t^{\des(w)} \\ 
&\ \underbrace{(1+t+t^2+\cdots)}_{\text{bars added in the $0$th space }} \underbrace{(1+q t+q^2 t^2+\cdots)}_{\text{bars added in the $1$st space }} \cdots\ \underbrace{(1+q^n t+q^{2n} t^2+\cdots)}_{\text{bars added in the $n$th space }} \\
=&\ q^{\maj(w)} t^{\des(w)} \frac{1}{1-t} \frac{1}{1-q^2 t} \cdots \frac{1}{1-q^n t} \\
=&\ \frac{q^{\maj(w)} t^{\des(w)}}{\prod_{i=0}^n (1-q^i t)}.
\end{align*}

We illustrate each step of the construction of barred permutations on $w$ by constructing $\beta = 1|\bar{1}|2|\bar{1}11||\bar{2}\bar{1}2$, a barred permutation of $w = 1\bar{1}2\bar{1}11\bar{2}\bar{1}2$. Here, the minimal barred permutation is $\tilde{\beta} = 1|\bar{1}2|\bar{1}11|\bar{2}\bar{1}2$.

Next, we count barred permutations with a fixed number of bars. 
For a signed permutation $w\in B_{\eta}$ we simplify notation by identifying a descent $i\in \Des(w)$ with its image $w(i)$.
We construct the barred permutation in $\mathbb{B}_{\eta}(k)$ by ‘putting $k$ bars in a line' and inserting $\eta_i$ elements of $\{\bar{i},i\}$ for all $1\leq i\leq r$ such that whenever there is a descent, there is a bar right after this position. Afterwards we compute the weight of the barred permutation we constructed.\\
Let $0\leq j\leq \eta_i$ denote the number of copies of $\bar{i}$ appearing in the barred permutation. For the $\eta_i-j$ copies of $i$ there are $k+1$ possible positions, namely on the left and right of every bar. Thus there are $\binom{\eta_i-j+k}{k}$ ways to allocate the $\eta_i-j$ copies of $i$. 
Another permutation statistic we make use of is the number of inversions of a multiset permutation $v$, defined by
\begin{align*}
\inv(v) := |\{ (i,j)\in [n]^2 : i<j \text{ and } v(i)>v(j) \}|.
\end{align*}
By a well-known interpretation of the $q$-binomial coefficient (see, for example, \cite[Proposition~1.7.1]{Stanley}) we have
\begin{align*}
\binom{\eta_i-j+k}{k}_q = \sum_{v\in S_{(k,\eta_i-j)}} q^{\inv(v)}.
\end{align*}
Consider the bijection between the set of barred permutations consisting of $k$ bars and $\eta_i-j$ copies of $i$ and $S_{(k,\eta_i-j)}$ by sending $i$ to 2 and each bar to $1$. Then for a fixed barred permutation $\beta$ and $v$ its image under this map we have
\begin{align*}
q^{\sum_{l=0}^{\eta_i} l\, b_l(\beta)} = q^{\inv(v)}.
\end{align*}
For example, to construct the barred permutation $1|\bar{1}|2|\bar{1}11||\bar{2}\bar{1}2$ we start with five bars in a line.
Inserting the three copies of $1$ between the five bars yields
\begin{align*}
1|||11||.
\end{align*}
This barred permutation has weight $q^{1 \cdot 3+3 \cdot 2} t^5 = q^9 t^5$. The corresponding permutation $21112211\in S_{(5,3)}$ has $9$ inversions.\\
By the definition of a barred permutation we observe the following:
\begin{itemize}
\item[(i)] There is no negative integer on the left of the first bar.
\item[(ii)] There is at most one negative integer between two bars.
\end{itemize}
We proceed with inserting negative elements, so we start again with $k$ bars in a line. 
Because of $(i)$ there are $k$ possible positions for each $\bar{i}$ and due to $(ii)$ there is at most one copy of $\bar{i}$ between two bars and on the right of the last bar. 
Thus there are $\binom{k}{j}$ ways to allocate the $\bar{i}$s. 
We fix one possible distribution and determine its weight. 
Constructing the barred permutation in the example above we obtain
\begin{align*}
| \bar{1} || \bar{1} || \bar{1}.
\end{align*}
We call a bar \emph{additional} if it does not immediately follow after a copy of $\bar{i}$ and if it is not the first bar appearing. Otherwise we call a bar \emph{nonadditional}. 
Clearly, there are $k-j$ additional bars.
For instance, the third and the last bar of $|\bar{1}||\bar{1}||\bar{1}$ are additional bars.\\
We divide the weight of the barred permutation into two factors
\begin{align*}
\wt(\beta) = q^{\sum_{l=0}^n l\, b_l(\beta)} t^{\sum_{l=0}^n b_l(\beta)} = q^{\sum_{l=0}^n l\, b_l^{\text{add}}(\beta)} t^{\sum_{l=0}^n b_l^{\text{add}}(\beta)} \cdot q^{\sum_{l=0}^n l\, b_l^{\text{non}}(\beta)} t^{\sum_{l=0}^n b_l^{\text{non}}(\beta)},
\end{align*}
where $b_l^{\text{add}}(\beta)$ is the number of additional bars in the $l$th space and $b_l^{\text{non}}(\beta)$ is the number of nonadditional bars in the $l$th space of $\beta$.
With a similar argument as above we describe the weight of the barred permutation in terms of inversions of the corresponding permutation in $S_{(j,k-j)}$: we obtain a multiset permutation $v\in S_{(j,k-j)}$ from a barred permutation $\beta$ by identifying $\bar{1}$ in $\beta$ with $2$ in $v$ and every additional bar with $1$ in $v$.
Then
\begin{align*}
q^{\sum_{l=0}^n l\, b_l^{\text{add}}(\beta)} t^{\sum_{l=0}^n b_l^{\text{add}}(\beta)} = \sum_{v\in S_{(j,k-j)}} q^{\inv(v)} t^{k-j} = \binom{k}{j}_q t^{k-j}.
\end{align*}
E.g.\ for $| \bar{1} || \bar{1} || \bar{1}$ the corresponding multiset permutation is $v = 21212\in S_{(3,5-3)}$ and its weight is given by $q^{1+2} t^2 = q^{\inv(21212)} t^2$.\\
For the second factor we take the nonadditional bars into account, which appear in the $0\text{th}, 1\text{st},\dots , (j-1)\text{th}$ space. Thus we obtain a factor of the form 
\begin{align*}
q^{\sum_{l=0}^{j-1} l\, b_l^{\text{non}}(\beta)} t^{\sum_{l=0}^{j-1} b_l^{\text{non}}(\beta)} = q^{\sum_{l=0}^{j-1} l} t^j = q^{\frac{j(j-1)}{2}} t^j.
\end{align*}
We proceed with inserting $\{\bar{i},i\}$ for all $1\leq i \leq r$ as described above in a way such that we insert the elements between two bars in an ascending order. 
In our example, we insert copies of $1,\bar{1},2$ and $\bar{2}$ between five bars: $1|||11||$, $|\bar{1}||\bar{1}||\bar{1}$, $||2|||2$, and $|||||\bar{2}$. By collecting all elements in the $i$th spaces for $0\leq i \leq n$ of the four barred permutations above and ordering them increasingly we obtain the barred permutation $1|\bar{1}|2|\bar{1}11||\bar{2}\bar{1}2$.\\
In general, this yields to
\begin{align*}
\sum_{\beta\in \mathbb{B}_{\eta}(k)} \wt(\beta) = \sum_{\beta\in \mathbb{B}_{\eta}(k)} q^{\sum_{l=0}^n l\, b_l(\beta)} t^k = \prod_{i=1}^r \sum_{j=0}^{\eta_i} \left( q^{\frac{j(j-1)}{2}} \binom{k}{j}_q \binom{\eta_i-j+k}{k}_q \right) t^k.
\end{align*}
Using the identity $\binom{n}{k}_q = \binom{n}{n-k}_q$ and rewriting $\binom{n}{k}_q = \frac{[n]_q!}{[n-k]_q![k]_q!}$ one easily sees that
\begin{align*}
\binom{k}{j}_q \binom{\eta_i -j+k}{k}_q &= \binom{k}{k-j}_q \binom{\eta_i -j+k}{k}_q = \frac{[\eta_i-j+k]_q!}{[j]_q! [k-j]_q! [\eta_i-j]_q!} \\ &= \binom{\eta_i}{j}_q \binom{k-j + \eta_i}{\eta_i}_q.
\end{align*}
The first equality of Theorem~\ref{MacMahonB} now follows from \eqref{barredpermUnion}:
\begin{align*}
\frac{\sum_{w\in B_{\eta}}q^{\maj(w)}t^{\des(w)}}{\prod_{i=0}^n(1-q^i t)} 
& = \sum_{w\in B_{\eta}} \sum_{\beta \in \mathbb{B}_{\eta}(w)} \wt(\beta)
= \sum_{k\geq 0} \sum_{\beta \in \mathbb{B}_{\eta}(k)} \wt(\beta) \\
& = \sum_{k\geq 0}  \left( \prod_{i=1}^r \sum_{j=0}^{\eta_i} \left( q^{\frac{j(j-1)}{2}} \binom{\eta_i}{j}_q \binom{k-j+\eta_i}{\eta_i}_q \right) \right) t^k.
\end{align*}

\textbf{Part 2:}
In the second part we start with the weighted Ehrhart series of an $n$-dimensional cross polytope and show that
\begin{equation}\label{eq:MacMahonB_identity2}
\sum_{k\geq 0} \left( \sum_{j= 0}^{n} \left( q^{\frac{j(j-1)}{2}} \binom{n}{j}_q \binom{k-j+n}{n}_q \right)\right) t^k= \Ehr_{\cp_n, \bar{\boldsymbol{\mu}}_n}(q,t).
\end{equation}
We generalise our results to products of cross polytopes later.

First, we subdivide the cross polytope into simplices.
This is analogous to the half-open triangulation in \cite{BeckBraun}, but since our weight function on $k\hspace{0,02cm} \cp_n$ is not a special case of integer point transform, but contains a bijection $\phi_{k,n}$, we can not abbreviate this step by using \cite{BeckBraun}.

For 
\begin{align*}
Q^J = \left\lbrace (x_1, \dots ,x_n)\in \mathbb{Z}^n : x_j<0\ \forall j\in J,\ 0\leq x_i \text{ otherwise} \right\rbrace
\end{align*}
for a subset $J\subseteq [n]$ we identify
\begin{align*}
k\hspace{0,02cm} \cp_n \cap Q^J = \left\lbrace (x_1,\dots ,x_n)\in \mathbb{Z}^n : |x_1| + \dots + |x_n| \leq k,\ x_j<0\ \forall j\in J,\ 0\leq x_i \text{ otherwise} \right\rbrace
\end{align*}
with 
\begin{align*}
\left( k\hspace{0,02cm} \cp_n \cap Q^J\right)^+ := \left\lbrace (x_1,\dots ,x_n)\in \mathbb{Z}^n : x_1 + \dots + x_n \leq k,\ 0 < x_j\ \forall j\in J,\ 0\leq x_i \text{ otherwise} \right\rbrace,
\end{align*}
which is contained in $k\Delta_n$.
For instance, the two sets are illustrated in Figure~\ref{fig:1} for $k=3$, $n=2$ and $J=\{1\}$. 
The lattice points in the red simplex in the cross polytope on the left hand side correspond to $k\hspace{0,02cm} \cp_n \cap Q^J$ and those in the shifted simplex in the cross polytope on the right hand side correspond to $\left( k\hspace{0,02cm} \cp_n \cap Q^J\right)^+$.

\begin{figure}
\centering
\begin{minipage}[b]{.33\textwidth} 
      \includegraphics[width=\textwidth]{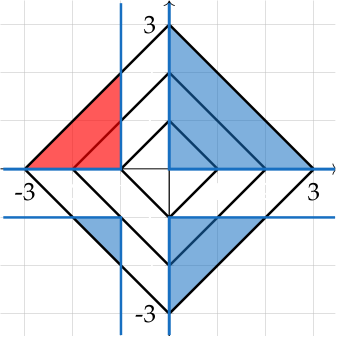}
   \end{minipage}
   \hspace{.045\linewidth}
   \begin{minipage}[b]{.33\textwidth} 
      \includegraphics[width=\textwidth]{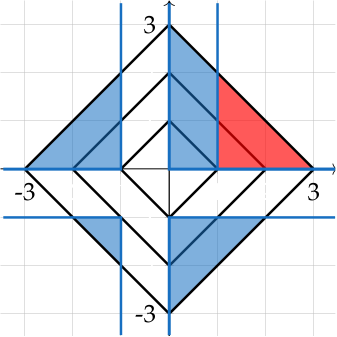}
   \end{minipage}
\caption{A subdivision of $3 \protect\cp_2 \cap \mathbb{Z}^2$
 into (shifted) standard simplices. The red simplex in the cross polytope on the left is identified with the red one on the right.}
\label{fig:1}
\end{figure}

\noindent Let $J=\{ j_1,\dots ,j_m \}$ be a subset of $[n]$ with $j_1<\dots <j_m$ and let further 
$\delta_{l\in J}$ denote the Kronecker delta, defined by $\delta_{l\in J}=1$ if $l\in J$ and zero otherwise. We compute the image of the above subset under $\phi_{k,n}$:
\begin{align*}
\phi_{k,n}&\left(\left(k\hspace{0,02cm} \cp_n \cap Q^J\right)^+\right)\\ &= \phi_{k,n} \Bigg{(} \Bigg{\{} (x_j+1, x_i)_{\substack{j\in J\\ i\in [n]\setminus J}}\in \mathbb{Z}^n : \sum_{j\in J}(x_j+1) + \sum_{i\in [n]\setminus J} x_i \leq k,\ 0 \leq x_j, x_i \Bigg{\}} \Bigg{)} \\
&=\Bigg{\{} \Bigg{(} k-\sum_{l=i}^n (x_l + \delta_{l\in J})\Bigg{)}_{i \in [n]} \in \mathbb{Z}^n : x_1 + \dots + x_n \leq k - m,\ 0 \leq x_i \Bigg{\}} \\
&= \Bigg{\{} \Bigg{(} \underbrace{k-m -\sum_{l=i}^n x_l}_{y_i} + m - \sum_{l=i}^n \delta_{l\in J}\Bigg{)}_{i \in [n]} \in \mathbb{Z}^n : x_1 + \dots + x_n \leq k - m,\ 0 \leq x_i \Bigg{\}} \\
&= \Bigg{\{} \Bigg{(} y_i + m - \sum_{l=i}^n \delta_{l\in J}\Bigg{)}_{i\in [n]} \in \mathbb{Z}^n : 0\leq y_1 \leq \dots \leq y_n \leq k-m \Bigg{\}} \\ 
&= \underbrace{\Bigg{(} m - \sum_{l=i}^n \delta_{l\in J}\Bigg{)}_{i\in [n]}}_{=: s^J} + (k-m) \mathcal{O}(\Delta_n).
\end{align*}
For $J$ as above the shift is given by 
\begin{align*}
s^J = (\underbrace{0,\dots ,0}_{j_1}, \underbrace{1,\dots ,1}_{j_2-j_1}, \dots , \underbrace{j-1, \dots ,j-1}_{j_m-j_{m-1}}, \underbrace{m, \dots ,m}_{n-j_m}).
\end{align*}
We encode this shift by 
\begin{align*}
v^J := (1,\dots ,1,2,1,\dots ,1,2, \dots , 1, \dots ,1,2, 1, \dots ,1)\in S_{(m,n-m)},
\end{align*}
where there is a $2$ in the $j_1\text{th},\dots , j_m\text{th}$ component and $1$ elsewhere.
We describe the sum of the coordinates of $s^J$ in terms of inversions (cf.~\cite[Proposition~1.7.1]{Stanley}) of $v^J$ as follows: 
\begin{align}\label{eq:lambda(shift)}
\begin{split}
\lambda_n (s^J) &= \inv(v^J) + 1 +2+ \dots + m-1\\
&= \binom{n}{m}_q + \frac{m(m-1)}{2}.
\end{split}
\end{align}
Summation over all $J\subseteq [n]$, \eqref{eq:lambda(shift)} and Example~\ref{exa_q-Ehrhartseries_Chapoton}~(i) yield
\begin{align*}
\sum_{x\in k \hspace{0,02cm} \cp_n \cap  \mathbb{Z}^n} q^{\bar{\boldsymbol{\mu}}_{k,n}(x)}
&= \sum_{m=0}^n \sum_{\substack{J\subseteq [n],\\ |J|=m}} q^{\lambda_n (s^J)} \sum_{x\in (k-m)\mathcal{O}(\Delta_n)}
q^{\lambda_n(x)} \\
&= \sum_{m=0}^n q^{\frac{m(m-1)}{2}} \binom{n}{m}_q \binom{k-m+n}{n}_q
\end{align*}
which proves \eqref{eq:MacMahonB_identity2}.\\

By \eqref{eq:2.2.2nr.2} the weight functions $\bar{\boldsymbol{\mu}}_{\eta}$ are compatible with taking products, so for $n= \sum_{i=1}^r \eta_i$ we obtain
\begin{align*}
\sum_{x\in k \hspace{0,02cm} \cp_{\eta} \cap  \mathbb{Z}^n} q^{\bar{\boldsymbol{\mu}}_{k,\eta}(x)}
= \prod_{i=1}^r \sum_{x_i\in k \hspace{0,02cm} \cp_{\eta_i} \cap \mathbb{Z}^{\eta_i}} q^{\bar{\boldsymbol{\mu}}_{k,\eta_i}(x)}
= \prod_{i=1}^r \sum_{m=0}^{\eta_i} q^{\frac{m(m-1)}{2}} \binom{\eta_i}{m}_q \binom{k-m+\eta_i}{\eta_i}_q.
\end{align*}
Therefore,
\begin{align*}
\Ehr_{\cp_{\eta}, \bar{\boldsymbol{\mu}}}(q,t) 
= \sum_{k\geq 0} \left( \prod_{i=1}^r \sum_{m=0}^{\eta_i} \left( q^{\frac{m(m-1)}{2}} \binom{\eta_i}{m}_q \binom{k-m+\eta_i}{\eta_i}_q \right) \right) t^k.
\end{align*}
Setting $m=j$ the second equality of Theorem~\ref{MacMahonB} is proven. 
\end{proof}

\section{Properties of the generalised Eulerian numbers of types~$\mathsf{A}$ and~$\mathsf{B}$}\label{section:Properties}

In this section we leverage our dictionary of generalised Eulerian polynomials on the one hand and $h^*$-polynomials of suitable polytopes on the other hand to establish additional properties of the generalised Eulerian numbers of types~$\mathsf{A}$ and~$\mathsf{B}$.
We use the special case of Theorems~\ref{MacMahonA} and~\ref{MacMahonB} where $q=1$ to re-prove palindromicity of the generalised Eulerian numbers of type~$\mathsf{A}$ and real-rootedness of the Eulerian numbers of types~$\mathsf{A}$ and~$\mathsf{B}$.
These properties are known by \cite[Proposition~2.12]{OrbitDirichlet}, \cite{Frobenius}, and \cite{Brenti}.
Moreover, we obtain new results for the generalised Eulerian numbers of type~$\mathsf{B}$, which turn out to be palindromic and unimodal.

\begin{rem} \label{rem_palindromicity}
A polynomial $h(t)= h_n t^n + \dots + h_0 \in \mathbb{Q}[t]$ is \emph{palindromic} if its coefficients are symmetric, i.e.\ $h_k = h_{n-k}$ for all $1\leq k \leq \lfloor \frac{n}{2} \rfloor$.
\begin{itemize}
\item[(a)] The generating polynomial of the joint distribution of major index and descent statistic over $S_{\eta}$ is palindromic if and only if $\eta$ is a rectangle, i.e.\ $\eta=(m,\dots,m)$ for some $m\in \mathbb{N}$.
In this case, \cite[Proposition~2.12]{OrbitDirichlet} states that 
\begin{align}\label{S_eta_palindromic}
C_{S_{\eta}}(q^{-1},t^{-1}) = q^{-m^2\binom{r}{2}} t^{-m(r-1)} C_{S_{\eta}}(q,t).
\end{align}
For $q=1$ this reflects the fact that the $h^*$-polynomial of a polytope is palindromic if and only if the polytope is Gorenstein: $\Delta_{\eta}$ is Gorenstein if and only if all simplices $\Delta_{\eta_i}$ have the same dimension, i.e.\ $\eta_i=\eta_j$ for all $i,j\in [r]$, see Example~\ref{exa_Gorenstein}~(i) and~(ii).
In this case, \eqref{S_eta_palindromic} for $q=1$ follows by Proposition~\ref{thm_Gorenstein_functional_eq}.
\item[(b)] In contrast: the generating polynomial of the joint distribution of major index and descent statistic over $B_{\eta}$ is palindromic for all $\eta$.
This follows immediately from the bijection which switches signs:
\begin{align*}
\psi: B_{\eta} &\to B_{\eta} \\
w^{\epsilon} &\mapsto w^{-\epsilon},
\end{align*}
where $-\epsilon(i):=(-1)\epsilon(i)$ for all $i\in [n]$.\\
Under this map, $\Des(\psi(w^{\epsilon}))=[n-1]_0\setminus \Des(w^{\epsilon})$, and thus
\begin{align}\label{B_eta_palindromic}
C_{B_{\eta}}(q^{-1},t^{-1}) = q^{-\binom{n}{2}} t^{-n} C_{B_{\eta}}(q,t).
\end{align}
For $q=1$ again this reflects the fact that $\cp_{\eta}$ is Gorenstein of index $1$, since every $\cp_{\eta_i}$ is, see Example~\ref{exa_Gorenstein}~(iii) and~(iv), and thus by Proposition~\ref{thm_Gorenstein_functional_eq} \eqref{B_eta_palindromic} holds for $q=1$.
\end{itemize}
\end{rem}

Next we prove unimodality of the generalised Eulerian numbers of type~$\mathsf{B}$ (equivalently, of the $h^*$-polynomial of products of cross polytopes), i.e.\ $$h^*_0 \leq \dots \leq h^*_{k-1} \leq h^*_k \geq h^*_{k+1} \geq \dots \geq h^*_n,$$ for $k= \lfloor \frac{n}{2}\rfloor \in\mathbb{N}$ by showing that products of cross polytopes are (reflexive and) what is called \emph{anti-blocking}. The $h^*$-polynomials of such polytopes are known to be unimodal by \cite[Theorem~3.4]{KohlOlsenSanyal} and \cite[Theorem~1]{BrunsRoemer}; see also \cite{BrunsRoemer} for the relevant definitions.

\begin{thm}[{\cite[Theorem~1]{BrunsRoemer}}]\label{BrunsRoemer_unimodal}
Let $\mathcal{P}$ be a Gorenstein polytope with a regular, unimodal triangulation. Then the $h^*$-polynomial of $\mathcal{P}$ is unimodal.
\end{thm}

Let $\mathcal{P}_+$ denote the intersection of $\mathcal{P}$ with $\mathbb{R}^n_+ :=\{ x\in\mathbb{R}^n : 0 \leq x_1,\dots ,x_n \}$, viz.\ the first orthant. The polytope $\mathcal{P}_+$ is \emph{anti-blocking} if, for any $x\in \mathcal{P}_+$, $y\in \mathbb{R}^n$ with $0\leq y_i\leq x_i$ for all $i$, we have $y\in \mathcal{P}_+$.
For $\sigma \in \{\pm 1\}^n$ and $x\in \mathbb{R}^n$ we denote by $\sigma x$ their componentwise product $(\sigma_1 x_1, \dots ,\sigma_n x_n)\in \mathbb{R}^n$.
A polytope $\mathcal{P}$ is \emph{locally anti-blocking} if $(\sigma\mathcal{P})\cap \mathbb{R}^n_+$ is anti-blocking for every $\sigma \in \{\pm 1\}^n$.

\begin{thm}[{\cite[Theorem~3.4]{KohlOlsenSanyal}}]\label{KohlOlsenSanyal_unimodal}
If $\mathcal{P}$ is a reflexive and locally anti-blocking polytope, then $\mathcal{P}$ has a regular, unimodal triangulation. In particular, the $h^*$-polynomial of $\mathcal{P}$ is unimodal.
\end{thm}

\begin{prop}
The generalised Eulerian numbers of type~$\mathsf{B}$ are unimodal.
In particular, the Eulerian numbers of type~$\mathsf{B}$ are unimodal.
\end{prop}
\begin{proof}
Clearly, the cross polytope is reflexive and locally anti-blocking: for every $\sigma \in \{\pm 1\}^n$, $\sigma\cp_n \cap \mathbb{R}^n_+ = \Delta_n$ and for $0\leq y_i \leq x_i$ with $y\in \mathbb{R}^n_+$ and $x\in\Delta_n$, it follows that $y\in\Delta_n$.
This extends to products of cross polytopes, so $\cp_{\eta}$ is locally anti-blocking. Thus, by Theorems~\ref{BrunsRoemer_unimodal} and~\ref{KohlOlsenSanyal_unimodal} the $h^*$-vector of $\cp_{\eta}$ is unimodal and so the generalised Eulerian numbers of type~$\mathsf{B}$ are unimodal.
\end{proof}

Further it is known that the Eulerian numbers (of type~$\mathsf{A}$) can be interpreted as the $h$-vector of the barycentric subdivision of the boundary of the simplex; cf.~\cite[Theorem~2.2]{BrentiWelker}. 
Interpreting the simplex as a type-$\mathsf{A}$ polytope and the cross polytope a type-$\mathsf{B}$ analogue is supported by the following proposition:
\begin{prop}[{\cite{Bjoerner},\cite[Theorem~2.3]{Brenti}}]
The $h$-vector of the barycentric subdivision of the boundary of the cross polytope is given by the Eulerian numbers of type~$\mathsf{B}$.
\end{prop}
\noindent As Brenti pointed out in \cite{BrentiWelker}, the above proposition is obtained from Theorems~1.6 and~2.1 and Proposition~1.2 in \cite{Bjoerner}.
Using \cite[Theorem~3.1]{BrentiWelker} this leads to the following corollary.
\begin{cor}
The Eulerian polynomials of types~$\mathsf{A}$ and~$\mathsf{B}$ only have real roots. In particular, the sequences of their coefficients are unimodal.
\end{cor}

Computations with SageMath~\cite{sage} show that the generalised Eulerian polynomials of types~$\mathsf{A}$ and~$\mathsf{B}$ only have real roots (at least up to $n=8$, see Appendizes~\ref{Liste_EulerianPolynomialsA} and~\ref{Liste_EulerianPolynomialsB}).
This leads to the following conjecture:
\begin{con}
The generalised Eulerian numbers of types~$\mathsf{A}$ and~$\mathsf{B}$ only have real roots.
\end{con}

\section{Further generalisations -- coloured multiset permutations}\label{section:Generalisation}

Multivariate generalisations of Example~\ref{exa-q-Ehrhartseries}~(ii), Corollary~\ref{Cor_hypercubein0}, and Corollary~\ref{Cor_colouredperm_cube} are developed in \cite{BeckBraun}. 
Considering the Coxeter-theoretic background of the descent polynomials of permutations and signed permutations, it may seem natural to seek a type-$\mathsf{D}$ analogue of MacMahon's formula. A first step towards this would be to find an $n$-dimensional polytope $\mathcal{P}_n$ such that
\begin{align*}
\frac{\sum_{w\in D_n} t^{\des(w)}}{(1-t)^{n+1}} = \Ehr_{\mathcal{P}_n}(t).
\end{align*}
However, at present, we do not know how to generalise elements in $D_n$ to even signed multiset permutations without losing the product structure of the corresponding Ehrhart series.\\

Another natural way for a generalisation comes from considering $B_n$ as the wreath product of the cyclic group of order two by the symmetric group. This leads to the study of coloured permutations $S_n^c:=\mathbb{Z}/c\mathbb{Z} \wr S_n$, for $c\in\mathbb{N}_0$, and, more generally, coloured multiset permutations.
In Proposition~\ref{prop_colouredmultisetperm} we show that, if the composition $\eta$ has only `small' parts, a `coloured MacMahon's formula' holds. 
This means that the descent polynomial over coloured (multiset) permutations can be interpreted as an $h^*$-polynomial of a polytope.\\
Similarly to the description of signed multiset permutations in terms of pairs $(w,\epsilon)$ of multiset permutations and signed vectors, we define coloured multiset permutations. 
We denote a \emph{coloured multiset permutation} $w^{\gamma}:=(w,\gamma)$ by the indexed permutation $w^{\gamma}~=~w_1^{\gamma_1} \cdots w_n^{\gamma_n}$, where $w\in S_{\eta}$ 
and $\gamma\colon [n] \to \{0,\dots,c-1\}$. We write $S_{\eta}^c$ for the set of all coloured multiset permutations.

\begin{dfn}\label{dfn:descent_coloured_perm}
Fixing the ordering $$r^{c-1} < \dots < 1^{c-1} < \dots < r^1 < \dots < 1^1 < 1^0 < \dots < r^0,$$ we define a \emph{descent} statistic as
\begin{align}\label{eq:des}
\des(w^{\gamma}) = | \{ i\in [n-1]_0 :\ & \gamma_i=\gamma_{i+1}=0 \text{ and } w_i> w_{i+1},\nonumber\\
 \text{ or } & \gamma_i=\gamma_{i+1}>0 \text{ and } w_i \leq w_{i+1}, \\ 
 \text{ or } & \gamma_i<\gamma_{i+1} \}|\nonumber,
\end{align}
where $w_0^{\gamma_0} := 0^0$.
\end{dfn}
For instance, for $\eta=(1,2)$ and $c=3$ the number of descents of $1^2 2^1 2^1 \in S_{(1,2)}^3$ is $\des(1^2 2^1 2^1)=2$, since $2^2 < 1^2 < 2^1 < 1^1 < 1^0 < 2^0$. 
In the special case of $c=2$ we obtain the descent statistic from Section~\ref{section:subsignedmultisetperm} and for $c=1$ this reduces to descents on multiset permutations, cf.\ Section~\ref{section:submultisetperm}.
To simplify notation we abbreviate a coloured permutation $w^{\gamma}$ to $w$.

\begin{dfn}\label{dfn:descentPolynomialColouredPerm}
For $c\in\mathbb{N}$ and $\eta$ a composition of some $n\in\mathbb{N}$ we denote by
\begin{align*}
d_{S_{\eta}^c}(t) := \sum_{w \in S_{\eta}^c} t^{\des(w)}
\end{align*}
the \emph{descent polynomial of $S_{\eta}^c$}.
\end{dfn}
For instance, for $\eta = (2)$ and $c=3$ we have
\begin{align*}
d_{S_{(2)}^3}(t) = 3t^2+5t+1.
\end{align*}
In a special case we are able to show that the descent polynomial of $S_{\eta}^c$ is an $h^*$-polynomial of a product of certain  polytopes, see Proposition~\ref{prop_colouredmultisetperm}. 
Intuitively, compared to $B_{\eta}$ we increase the number of negatives by adding colours. 
This leads to a product of polytopes of the following type:
\begin{dfn}\label{dfn:distorted_cp}
For $c,n\in \mathbb{N}$ we denote by
\begin{align*}
\mathcal{C}_n^c := \conv \{e_1,\dots,e_n, -(c-1) e_1, \dots ,-(c-1) e_n\}
\end{align*}
a \emph{distorted cross polytope} in $\mathbb{R}^n$.
\end{dfn}
\noindent As an example, $\mathcal{C}_2^3$ is illustrated in Figure~\ref{fig:2}. 

\begin{figure}
\centering
\includegraphics[width=0.33\textwidth]{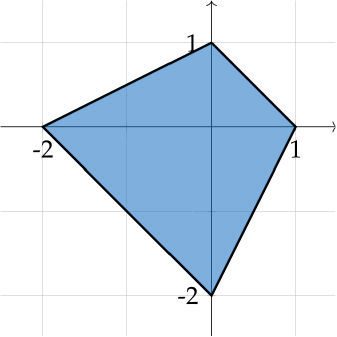}
\caption{The distorted cross polytope $\mathcal{C}_2^3$.}
\label{fig:2}
\end{figure}

\begin{pro}\label{Ques1}
Formulate a permutation statistic $\stat \colon S_{\eta}^c \to \mathbb{N}_0$ such that
\begin{align}\label{eq:ques}
\frac{\sum_{w \in S_{\eta}^c} t^{\stat(w)}}{(1-t)^{n+1}} = \Conv_{i=1}^r \Ehr_{\mathcal{C}_{\eta_i}^c} (t).
\end{align}
\end{pro}

\noindent In the special case of $\eta = (\eta_1,\dots ,\eta_r)$ with $\eta_i\leq 2$ for every $i$, we give such a formulation as in Problem~\ref{Ques1}.
\begin{prop}\label{prop_colouredmultisetperm}
For $\eta = (\eta_1,\dots ,\eta_r)$ with $\eta_i\leq 2$ for every $i$, we have
\begin{align} \label{eq:prop}
\frac{d_{S_{\eta}^c}(t)}{(1-t)^{n+1}} = \Conv_{i=1}^r \Ehr_{\mathcal{C}_{\eta_i}^c} (t).
\end{align}
More precisely, for $\eta = (\underbrace{1,\dots ,1}_{l},\underbrace{2,\dots ,2}_{r-l})$ for some $l \in [r]_0$ we obtain
\begin{align*}
\frac{\sum_{w\in S_{\eta}^c} t^{\des(w)}}{(1-t)^{n+1}} = \sum_{k\geq 0} (ck+1)^l \left( \frac{c^2}{2}k^2 + \frac{c+2}{2}k +1 \right)^{r-l} t^k 
= \Ehr_{(\mathcal{C}_1^c)^l\times (\mathcal{C}_2^c)^{r-l}}(t).
\end{align*}
\end{prop}
\noindent The proof is omitted. It uses barred permutations and is similar to the one of Theorem~\ref{MacMahonB}.

For instance, for $\eta=(1,2)$ and $c=3$ the corresponding polytope is $\mathcal{C}_1^3\times \mathcal{C}_2^3$, which is a cylinder of height three over $\mathcal{C}_2^3$.
In the special case where $S_{\eta}^c = S_n^c$ the corresponding polytope is the $c$th dilate of the $n$-dimensional unit cube. 
\begin{cor}\label{Cor_colouredperm_cube}
Setting $\eta=(1,\dots, 1)$ and therefore $S_{\eta}^c=S_n^c$ in Proposition~\ref{prop_colouredmultisetperm} leads to the identity
\begin{align*}
\frac{d_{S_n^c}(t)}{(1-t)^{n+1}} = \Conv_{i=1}^n \Ehr_{\mathcal{C}_1^c} (t) = \sum_{k\geq 0} (ck+1)^n t^k = \Ehr_{c\hspace{0.05cm} \Box_n}(t).
\end{align*}
\end{cor}

\begin{rem}
The identity in Corollary~\ref{Cor_colouredperm_cube} appears also for other descent statistics on coloured permutations, see \cite[Section~2]{BeckBraun} and \cite[Theorem~17 and Theorem~32]{Steingrimsson}. Therefore, even if the definition of the descent statistic in Definition~\ref{dfn:descent_coloured_perm} differs from the one in \cite{BeckBraun} and \cite{Steingrimsson}, Corollary~\ref{Cor_colouredperm_cube} shows that all three statistics have the same distribution over coloured permutations.
\end{rem}

Another formula we are able to prove is the following: consider the set of coloured permutations, viz.\ $\eta = (1,\dots ,1)$, but change the colouring, such that we allow the colours $0,\dots , c_1-1$ for the letter $1$, the colours $0,\dots , c_2-1$ for $2$, and so on, where the permutation is written in one-line notation and $(\boldsymbol{c}) := (c_1,\dots , c_n) \in \mathbb{N}^n$; see, e.g.,~\cite[Section~6.2]{FerroniMcGinnis}. We denote the set of all so-called \emph{$(\boldsymbol{c})$-coloured permutations} by $S_n^{(\boldsymbol{c})}$. For instance, for $n=2$ and $(\boldsymbol{c})=(c_1,c_2)=(2,3)$ we obtain
\begin{align*}
S_2^{(2,3)}= \{ 1^02^0, 1^02^1, 1^02^2, 1^12^0, 1^12^1, 1^12^2, 2^01^0, 2^01^1, 2^01^2, 2^11^0, 2^11^1, 2^11^2 \}.
\end{align*}
In general, we can associate the set of $(\boldsymbol{c})$-coloured permutations with a product of one-dimensional distorted cross polytopes:
\begin{prop}
The descent polynomial over the set of $(\boldsymbol{c})$-coloured permutations is the $h^*$-polynomial of the $n$-dimensional hyperrectangle with edge length $c_1,\dots , c_n$. More precisely
\begin{align*}
\frac{\sum_{w \in S_n^{(\boldsymbol{c})}} t^{\des(w)}}{(1-t)^{n+1}} = \sum_{k\geq 0} \left( \prod_{i=1}^n (c_ik+1) \right) t^k = \Conv_{i=1}^n \Ehr_{\mathcal{C}_1^{c_i}} (t) = \Ehr_{\left(\mathcal{C}_1^{c_i}\right)^n}(t).
\end{align*}
\end{prop}
\noindent The proof is analogous to the one of Theorem~\ref{MacMahonB}.

\begin{rem}
It easily follows from the definition of the descent set that $d_{S_{\eta}^c}$ is palindromic if and only if $c=2$, i.e.\ $S_{\eta}^c =B_{\eta}$ or $c=1$, i.e.\ $S_{\eta}^c =S_{\eta}$ and $\eta$ is a rectangle.
Similarly, $d_{S_n^{(\textbf{c})}}$ is palindromic if and only if $(\textbf{c})=(1,\dots ,1)$, i.e.\ $S_n^{(\textbf{c})} = S_n$. 
\end{rem}

The statement of~\eqref{eq:prop} in Proposition~\ref{prop_colouredmultisetperm} fails for larger $\eta$, even for $\eta = (3)$. Computations with SageMath~\cite{sage} show that, in general, the descent polynomial defined by \eqref{eq:des} is not an $h^*$-polynomial of a polytope.
For $\eta= (6)$ and $c=5$ the descent polynomial on coloured multiset permutations is given by 
\begin{align*}
84 t^6 + 1920 t^5 + 6685 t^4 + 5609 t^3 + 1253 t^2 + 73 t + 1.
\end{align*}
This polynomial is not an $h^*$-polynomial of a polytope $\mathcal{P}$.
Indeed, by Lemma~3.13 and Corollary~3.16 in \cite{BeckRobins} the leading coefficient of the $h^*$-polynomial $h^*(t)= h_n^* t^n + \dots + h_1^* t + h_0^*$ satisfies $h_n^* = \Lp_{\mathcal{P}^{\circ}}(1)\leq \Lp_{\mathcal{P}}(1) -n-1 = h_1^*$, which is not the case for the descent polynomial as defined in Definition~\ref{dfn:descent_coloured_perm}.\\
At least a necessary condition for Problem~\ref{Ques1} to be solved is satisfied:
a lemma in Ehrhart theory \cite[Corollary~3.21]{BeckRobins} states that the coefficients of the $h^*$-polynomial sum up to the normalized volume of the $n$-dimensional polytope, i.e.\ $h_n^*+\dots +h_0^*=~n!\,\vol(\mathcal{P})$, which leads to the following proposition:
\begin{prop}\label{prop:numerator:t=1}
For $t=1$, the numerators of both sides in \eqref{eq:ques} coincide. 
More precisely,
\begin{align*}
|S_{\eta}^c| = n!\, \vol\left(\mathcal{C}_{\eta}^c\right) = c^n \frac{n!}{\eta_1!\dots \eta_r!},
\end{align*}
where $\mathcal{C}_{\eta}^c = \mathcal{C}_{\eta_1}^c \times \dots \times \mathcal{C}_{\eta_r}^c$.
\end{prop}
\begin{proof}
Clearly, $|S_{\eta}^c| = c^n \cdot |S_{\eta}|  = c^n \frac{n!}{\eta_1!\dots \eta_r!}$.
On the other side, $\mathcal{C}_{\eta_i}^c$ consists of $2^{\eta_i}$ simplices.
Let $S$ be an $l$-element subset of $[\eta_i]$, $1\leq l \leq \eta_i$.
There exists a simplex in $\mathcal{C}_{\eta_i}^c$ which is the convex hull of the union $\{e_i : i\in [\eta_i]\setminus S\} \cup \{ (c-1)(-e_i) : i\in S \}$ with volume $\frac{(c-1)^l}{n!}$.
For each $l$ there are $\binom{\eta_i}{l}$ simplices with volume $\frac{(c-1)^l}{n!}$. 
Thus
\begin{align*}
\vol(\mathcal{C}_{\eta_i}^c) = \sum_{l=0}^{\eta_i} \binom{\eta_i}{l} \frac{(c-1)^l}{\eta_i!} = \frac{c^{\eta_i}}{\eta_i!}.
\end{align*}
For the product $\mathcal{C}_{\eta}^c$ we obtain
\begin{align*}
\vol(\mathcal{C}_{\eta}^c) = \prod_{i=0}^r \frac{c^{\eta_i}}{\eta_i!} = \frac{c^n}{\eta_1! \dots \eta_r!}
\end{align*}
and therefore,
\begin{align*}
|S_{\eta}^c| = c^n \frac{n!}{\eta_1!\dots \eta_r!} = n! \vol\left(\mathcal{C}_{\eta}^c\right)
\end{align*}
as required.
\end{proof}
Proposition~\ref{prop:numerator:t=1} gives a hint that in Problem~\ref{Ques1} the distorted cross polytope might be the right polytope on the right hand side but the statistic on the left hand side might not be the one we are looking for.

\acknowledgements{The paper is part of the author's dissertation, supervised by Christopher Voll.
We would like to thank Angela Carnevale, Johannes Krah, and Christopher Voll for helpful discussions.\\
This version of the article has been accepted for publication after peer review but is not the Version of Record and does not reflect post-acceptance improvements, or any corrections. The Version of Record is available online at: \mbox{http://dx.doi.org/10.1007/s00454-025-00731-8}.}

\printbibliography

\newpage
\newgeometry{width=150mm,top=30mm,bottom=31mm,marginparwidth=2.5cm}
\appendix
\section{Generalised Eulerian polynomials of type~$\mathsf{A}$}\label{Liste_EulerianPolynomialsA}
\small
We list the generalised Eulerian polynomials of type~$\mathsf{A}$ for partitions $\eta = (\eta_1,\dots ,\eta_r)$ with $n=\sum_{i=1}^r \eta_i \leq 8$, see Definition~\ref{dfn:descentPolynomialS_eta}.
%\enlargethispage{3\baselineskip}
\scriptsize
\begin{longtable}{ccc}
\centering
$n$ &  $\eta$ & $d_{S_{\eta}}(t)$ \\
 \hline \hline
\vspace{0.1cm}$2$  & $(1,1)$ & $t + 1$ \\ 
$3$ & $(1,2)$ & $2 t + 1$ \\
\vspace{0.1cm} & $(1,1,1)$ & $t^2 + 4 t + 1$ \\ 
$4$ %& $(4)$ & $t^4 + 4 t^3 + 6 t^2 + 4 t + 1$ \\
 & $(1,3)$ & $3 t + 1$ \\
 & $(2,2)$ & $t^2 + 4 t + 1$ \\ 
 & $(1,1,2)$ & $4 t^2 + 7 t + 1$ \\
\vspace{0.1cm} & $(1,1,1,1)$ & $t^3 + 11 t^2 + 11 t + 1$ \\
$5$ %& $(5)$ & $t^5 + 5 t^4 + 10 t^3 + 10 t^2 + 5 t + 1$ \\ 
 & $(1,4)$ & $4 t + 1$ \\
 & $(2,3)$ & $3 t^2 + 6 t + 1$ \\
 & $(1,1,3)$ & $9 t^2 + 10 t + 1$ \\
 & $(1,2,2)$ & $2 t^3 + 15 t^2 + 12 t + 1$ \\
 & $(1,1,1,2)$ & $8 t^3 + 33 t^2 + 18 t + 1$ \\
\vspace{0.1cm} & $(1,1,1,1,1)$ & $t^4 + 26 t^3 + 66 t^2 + 26 t + 1$ \\
$6$ %& $(6)$ & $t^6 + 6 t^5 + 15 t^4 + 20 t^3 + 15 t^2 + 6 t + 1$ \\ 
 & $(1,5)$ & $5 t + 1$ \\
 & $(2,4)$ & $6 t^2 + 8 t + 1$ \\
 & $(3,3)$ & $t^3 + 9 t^2 + 9 t + 1$ \\
 & $(1,1,4)$ & $16 t^2 + 13 t + 1$ \\
 & $(1,2,3)$ & $9 t^3 + 33 t^2 + 17 t + 1$ \\
 & $(2,2,2)$ & $t^4 + 20 t^3 + 48 t^2 + 20 t + 1$ \\
 & $(1,1,1,3)$ & $27 t^3 + 67 t^2 + 25 t + 1$ \\
 & $(1,1,2,2)$ & $4 t^4 + 53 t^3 + 93 t^2 + 29 t + 1$ \\
 & $(1,1,1,1,2)$ & $16 t^4 + 131 t^3 + 171 t^2 + 41 t + 1$ \\ 
\vspace{0.1cm} & $(1,1,1,1,1,1)$ & $t^5 + 57 t^4 + 302 t^3 + 302 t^2 + 57 t + 1$ \\ 
$7$ %& $(7)$ & $t^7 + 7 t^6 + 21 t^5 + 35 t^4 + 35 t^3 + 21 t^2 + 7 t + 1$ \\ 
 & $(1,6)$ & $6 t + 1$ \\
 & $(2,5)$ & $10 t^2 + 10 t + 1$ \\
 & $(3,4)$ & $4 t^3 + 18 t^2 + 12 t + 1$ \\
 & $(1,1,5)$ & $25 t^2 + 16 t + 1$ \\
 & $(1,2,4)$ & $24 t^3 + 58 t^2 + 22 t + 1$ \\
 & $(1,3,3)$ & $3 t^4 + 40 t^3 + 72 t^2 + 24 t + 1$ \\
 & $(2,2,3)$ & $9 t^4 + 72 t^3 + 100 t^2 + 28 t + 1$ \\
 & $(1,1,1,4)$ & $64 t^3 + 113 t^2 + 32 t + 1$ \\
 & $(1,1,2,3)$ & $27 t^4 + 168 t^3 + 184 t^2 + 40 t + 1$ \\
 & $(1,2,2,2)$ & $2 t^5 + 65 t^4 + 272 t^3 + 244 t^2 + 46 t + 1$ \\
 & $(1,1,1,1,3)$ & $81 t^4 + 376 t^3 + 326 t^2 + 56 t + 1$ \\ 
 & $(1,1,1,2,2)$ & $8 t^5 + 179 t^4 + 584 t^3 + 424 t^2 + 64 t + 1$ \\
 & $(1,1,1,1,1,2)$ & $32 t^5 + 473 t^4 + 1208 t^3 + 718 t^2 + 88 t + 1$ \\ 
\vspace{0.1cm} & $(1,1,1,1,1,1,1)$ & $t^6 + 120 t^5 + 1191 t^4 + 2416 t^3 + 1191 t^2 + 120 t + 1$ \\
$8$ & $(1,7)$ & $7 t + 1$ \\
 & $(2,6)$ & $15 t^2 + 12 t + 1$ \\
 & $(3,5)$ & $10 t^3 + 30 t^2 + 15 t + 1$ \\
 & $(4,4)$ & $t^4 + 16 t^3 + 36 t^2 + 16 t + 1$ \\
 & $(1,1,6)$ & $36 t^2 + 19 t + 1$ \\
 & $(1,2,5)$ & $50 t^3 + 90 t^2 + 27 t + 1$ \\
 & $(1,3,4)$ & $16 t^4 + 106 t^3 + 126 t^2 + 31 t + 1$ \\
 & $(2,2,4)$ & $36 t^4 + 176 t^3 + 171 t^2 + 36 t + 1$ \\
 & $(2,3,3)$ & $3 t^5 + 69 t^4 + 244 t^3 + 204 t^2 + 39 t + 1$ \\
 & $(1,1,1,5)$ & $125 t^3 + 171 t^2 + 39 t + 1$ \\
 & $(1,1,2,4)$ & $96 t^4 + 386 t^3 + 306 t^2 + 51 t + 1$ \\
 & $(1,1,3,3)$ & $9 t^5 + 175 t^4 + 520 t^3 + 360 t^2 + 55 t + 1$ \\
 & $(1,2,2,3)$ & $27 t^5 + 333 t^4 + 788 t^3 + 468 t^2 + 63 t + 1$ \\
 & $(2,2,2,2)$ & $t^6 + 72 t^5 + 603 t^4 + 1168 t^3 + 603 t^2 + 72 t + 1$ \\
 & $(1,1,1,1,4)$ & $256 t^4 + 821 t^3 + 531 t^2 + 71 t + 1$ \\ 
 & $(1,1,1,2,3)$ & $81 t^5 + 807 t^4 + 1592 t^3 + 792 t^2 + 87 t + 1$ \\
 & $(1,1,2,2,2)$ & $4 t^6 + 207 t^5 + 1413 t^4 + 2308 t^3 + 1008 t^2 + 99 t + 1$ \\
 & $(1,1,1,1,1,3)$ & $243 t^5 + 1909 t^4 + 3134 t^3 + 1314 t^2 + 119 t + 1$ \\
 & $(1,1,1,1,2,2)$ & $16 t^6 + 585 t^5 + 3231 t^4 + 4456 t^3 + 1656 t^2 + 135 t + 1$ \\ 
 & $(1,1,1,1,1,1,2)$ & $64 t^6 + 1611 t^5 + 7197 t^4 + 8422 t^3 + 2682 t^2 + 183 t + 1$ \\
\vspace{0.1cm} & $(1,1,1,1,1,1,1,1)$ & $t^7 + 247 t^6 + 4293 t^5 + 15619 t^4 + 15619 t^3 + 4293 t^2 + 247 t + 1$\\
 $n$ & $(n)$ & $1$
\end{longtable}
\normalsize
\restoregeometry
\section{Generalised Eulerian polynomials of type~$\mathsf{B}$ and 
joint distributions of major index and descent over $B_{\eta}$}\label{Liste_EulerianPolynomialsB}

\small
Recall the definitions of the generalised Eulerian polynomials of type~$\mathsf{B}$ and the generating polynomial of the joint distribution of major index and descent statistic over $B_{\eta}$, see Definitions~\ref{dfn:descentPolynomialB_eta} and~\ref{dfn:Carlitz-q-Eulerian-polynomialB_eta}.
We list these polynomials for partitions $\eta = (\eta_1,\dots ,\eta_r)$ with $n=\sum_{i=1}^r \eta_i \leq 8$ and $n\leq 5$, respectively.

\scriptsize
\begin{longtable}{ccc}
\centering
$n$ &  $\eta$ & $d_{B_{\eta}}(t)$ \\
 \hline \hline
\vspace{0.1cm}$1$ & $(1)$ & $t+1$ \\
$2$ & $(2)$ & $t^2 + 2 t + 1$ \\
\vspace{0.1cm} & $(1,1)$ & $t^2 + 6 t + 1$ \\ 
$3$ & $(3)$ & $t^3 + 3 t^2 + 3 t + 1$ \\
 & $(1,2)$ & $t^3 + 11 t^2 + 11 t + 1$ \\
\vspace{0.1cm} & $(1,1,1)$ & $t^3 + 23 t^2 + 23 t + 1$ \\ 
$4$ & $(4)$ & $t^4 + 4 t^3 + 6 t^2 + 4 t + 1$ \\
 & $(1,3)$ & $t^4 + 16 t^3 + 30 t^2 + 16 t + 1$ \\
 & $(2,2)$ & $t^4 + 20 t^3 + 54 t^2 + 20 t + 1$ \\ 
 & $(1,1,2)$ & $t^4 + 40 t^3 + 110 t^2 + 40 t + 1$ \\
\vspace{0.1cm} & $(1,1,1,1)$ & $t^4 + 76 t^3 + 230 t^2 + 76 t + 1$ \\
$5$ & $(5)$ & $t^5 + 5 t^4 + 10 t^3 + 10 t^2 + 5 t + 1$ \\ 
 & $(1,4)$ & $t^5 + 21 t^4 + 58 t^3 + 58 t^2 + 21 t + 1$ \\
 & $(2,3)$ & $t^5 + 29 t^4 + 130 t^3 + 130 t^2 + 29 t + 1$ \\
 & $(1,1,3)$ & $t^5 + 57 t^4 + 262 t^3 + 262 t^2 + 57 t + 1$ \\
 & $(1,2,2)$ & $t^5 + 69 t^4 + 410 t^3 + 410 t^2 + 69 t + 1$ \\
 & $(1,1,1,2)$ & $t^5 + 129 t^4 + 830 t^3 + 830 t^2 + 129 t + 1$ \\
\vspace{0.1cm} & $(1,1,1,1,1)$ & $t^5 + 237 t^4 + 1682 t^3 + 1682 t^2 + 237 t + 1$ \\
$6$ & $(6)$ & $t^6 + 6 t^5 + 15 t^4 + 20 t^3 + 15 t^2 + 6 t + 1$ \\ 
 & $(1,5)$ & $t^6 + 26 t^5 + 95 t^4 + 140 t^3 + 95 t^2 + 26 t + 1$ \\
 & $(2,4)$ & $t^6 + 38 t^5 + 239 t^4 + 404 t^3 + 239 t^2 + 38 t + 1$ \\
 & $(3,3)$ & $t^6 + 42 t^5 + 303 t^4 + 588 t^3 + 303 t^2 + 42 t + 1$ \\
 & $(1,1,4)$ & $t^6 + 74 t^5 + 479 t^4 + 812 t^3 + 479 t^2 + 74 t + 1$ \\
 & $(1,2,3)$ & $t^6 + 98 t^5 + 911 t^4 + 1820 t^3 + 911 t^2 + 98 t + 1$ \\
 & $(2,2,2)$ & $t^6 + 118 t^5 + 1343 t^4 + 2836 t^3 + 1343 t^2 + 118 t + 1$ \\
 & $(1,1,1,3)$ & $t^6 + 182 t^5 + 1823 t^4 + 3668 t^3 + 1823 t^2 + 182 t + 1$ \\
 & $(1,1,2,2)$ & $t^6 + 218 t^5 + 2671 t^4 + 5740 t^3 + 2671 t^2 + 218 t + 1$ \\
 & $(1,1,1,1,2)$ & $t^6 + 398 t^5 + 5311 t^4 + 11620 t^3 + 5311 t^2 + 398 t + 1$ \\ 
\vspace{0.1cm} & $(1,1,1,1,1,1)$ & $t^6 + 722 t^5 + 10543 t^4 + 23548 t^3 + 10543 t^2 + 722 t + 1$ \\ 
$7$ & $(7)$ & $t^7 + 7 t^6 + 21 t^5 + 35 t^4 + 35 t^3 + 21 t^2 + 7 t + 1$ \\ 
 & $(1,6)$ & $t^7 + 31 t^6 + 141 t^5 + 275 t^4 + 275 t^3 + 141 t^2 + 31 t + 1$ \\
 & $(2,5)$ & $t^7 + 47 t^6 + 381 t^5 + 915 t^4 + 915 t^3 + 381 t^2 + 47 t + 1$ \\
 & $(3,4)$ & $t^7 + 55 t^6 + 549 t^5 + 1635 t^4 + 1635 t^3 + 549 t^2 + 55 t + 1$ \\
 & $(1,1,5)$ & $t^7 + 91 t^6 + 761 t^5 + 1835 t^4 + 1835 t^3 + 761 t^2 + 91 t + 1$ \\
 & $(1,2,4)$ & $t^7 + 127 t^6 + 1613 t^5 + 4979 t^4 + 4979 t^3 + 1613 t^2 + 127 t + 1$ \\
 & $(1,3,3)$ & $t^7 + 139 t^6 + 1977 t^5 + 6843 t^4 + 6843 t^3 + 1977 t^2 + 139 t + 1$ \\
 & $(2,2,3)$ & $t^7 + 167 t^6 + 2853 t^5 + 10419 t^4 + 10419 t^3 + 2853 t^2 + 167 t + 1$ \\
 & $(1,1,1,4)$ & $t^7 + 235 t^6 + 3209 t^5 + 9995 t^4 + 9995 t^3 + 3209 t^2 + 235 t + 1$ \\
 & $(1,1,2,3)$ & $t^7 + 307 t^6 + 5633 t^5 + 20939 t^4 + 20939 t^3 + 5633 t^2 + 307 t + 1$ \\
 & $(1,2,2,2)$ & $t^7 + 367 t^6 + 8013 t^5 + 31939 t^4 + 31939 t^3 + 8013 t^2 + 367 t + 1$ \\
 & $(1,1,1,1,3)$ & $t^7 + 559 t^6 + 11117 t^5 + 42083 t^4 + 42083 t^3 + 11117 t^2 + 559 t + 1$ \\ 
 & $(1,1,1,2,2)$ & $t^7 + 667 t^6 + 15753 t^5 + 64219 t^4 + 64219 t^3 + 15753 t^2 + 667 t + 1$ \\
 & $(1,1,1,1,1,2)$ & $t^7 + 1207 t^6 + 30933 t^5 + 129139 t^4 + 129139 t^3 + 30933 t^2 + 1207 t + 1$ \\ 
\vspace{0.1cm} & $(1,1,1,1,1,1,1)$ & $t^7 + 2179 t^6 + 60657 t^5 + 259723 t^4 + 259723 t^3 + 60657 t^2 + 2179 t + 1$ \\
$8$ & $(8)$ & $t^8 + 8 t^7 + 28 t^6 + 56 t^5 + 70 t^4 + 56 t^3 + 28 t^2 + 8 t + 1$ \\ 
 & $(1,7)$ & $t^8 + 36 t^7 + 196 t^6 + 476 t^5 + 630 t^4 + 476 t^3 + 196 t^2 + 36 t + 1$ \\
 & $(2,6)$ & $t^8 + 56 t^7 + 556 t^6 + 1736 t^5 + 2470 t^4 + 1736 t^3 + 556 t^2 + 56 t + 1$ \\
 & $(3,5)$ & $t^8 + 68 t^7 + 868 t^6 + 3516 t^5 + 5430 t^4 + 3516 t^3 + 868 t^2 + 68 t + 1$ \\
 & $(4,4)$ & $t^8 + 72 t^7 + 988 t^6 + 4344 t^5 + 7110 t^4 + 4344 t^3 + 988 t^2 + 72 t + 1$ \\
 & $(1,1,6)$ & $t^8 + 108 t^7 + 1108 t^6 + 3476 t^5 + 4950 t^4 + 3476 t^3 + 1108 t^2 + 108 t + 1$ \\
 & $(1,2,5)$ & $t^8 + 156 t^7 + 2516 t^6 + 10596 t^5 + 16470 t^4 + 10596 t^3 + 2516 t^2 + 156 t + 1$ \\
 & $(1,3,4)$ & $t^8 + 180 t^7 + 3460 t^6 + 17484 t^5 + 29430 t^4 + 17484 t^3 + 3460 t^2 + 180 t + 1$ \\
 & $(2,2,4)$ & $t^8 + 216 t^7 + 4940 t^6 + 26280 t^5 + 44646 t^4 + 26280 t^3 + 4940 t^2 + 216 t + 1$ \\
 & $(2,3,3)$ & $t^8 + 236 t^7 + 5956 t^6 + 34836 t^5 + 61302 t^4 + 34836 t^3 + 5956 t^2 + 236 t + 1$ \\
 & $(1,1,1,5)$ & $t^8 + 288 t^7 + 4988 t^6 + 21216 t^5 + 33030 t^4 + 21216 t^3 + 4988 t^2 + 288 t + 1$ \\
 & $(1,1,2,4)$ & $t^8 + 396 t^7 + 9716 t^6 + 52596 t^5 + 89622 t^4 + 52596 t^3 + 9716 t^2 + 396 t + 1$ \\
 & $(1,1,3,3)$ & $t^8 + 432 t^7 + 11692 t^6 + 69648 t^5 + 123174 t^4 + 69648 t^3 + 11692 t^2 + 432 t + 1$ \\
 & $(1,2,2,3)$ & $t^8 + 516 t^7 + 16436 t^6 + 104316 t^5 + 187542 t^4 + 104316 t^3 + 16436 t^2 + 516 t + 1$ \\
 & $(2,2,2,2)$ & $t^8 + 616 t^7 + 22972 t^6 + 155992 t^5 + 285958 t^4 + 155992 t^3 + 22972 t^2 + 616 t + 1$ \\
 & $(1,1,1,1,4)$ & $t^8 + 720 t^7 + 19100 t^6 + 105264 t^5 + 179910 t^4 + 105264 t^3 + 19100 t^2 + 720 t + 1$ \\ 
 & $(1,1,1,2,3)$ & $t^8 + 936 t^7 + 32156 t^6 + 208536 t^5 + 376902 t^4 + 208536 t^3 + 32156 t^2 + 936 t + 1$ \\
 & $(1,1,2,2,2)$ & $t^8 + 1116 t^7 + 44836 t^6 + 311716 t^5 + 574902 t^4 + 311716 t^3 + 44836 t^2 + 1116 t + 1$ \\
 & $(1,1,1,1,1,3)$ & $t^8 + 1692 t^7 + 62852 t^6 + 416868 t^5 + 757494 t^4 + 416868 t^3 + 62852 t^2 + 1692 t + 1$ \\
 & $(1,1,1,1,2,2)$ & $t^8 + 2016 t^7 + 87436 t^6 + 622816 t^5 + 1155942 t^4 + 622816 t^3 + 87436 t^2 + 2016 t + 1$ \\ 
 & $(1,1,1,1,1,1,2)$ & $t^8 + 3636 t^7 + 170356 t^6 + 1244236 t^5 + 2324502 t^4 + 1244236 t^3 + 170356 t^2 + 3636 t + 1$ \\
 & $(1,1,1,1,1,1,1,1)$ & $t^8 + 6552 t^7 + 331612 t^6 + 2485288 t^5 + 4675014 t^4 + 2485288 t^3 + 331612 t^2 + 6552 t + 1$ \\
& & \\
$n$ &  $\eta$ & $C_{B_{\eta}}(q,t)$ \\
 \hline \hline
\vspace{0.1cm}$1$ & $(1)$ & $t + 1$ \\
$2$ & $(2)$ & $q t^2 + (q+1) t + 1$ \\
\vspace{0.1cm} & $(1,1)$ & $q t^2 + (3 q + 3) t + 1$ \\
$3$ & $(3)$ & $q^3 t^3 + (q^3 + q^2 + q) t^2 + (q^2 + q + 1) t + 1$ \\
 & $(1,2)$ & $q^3 t^3 + (3q^3+5q^2+3q)t^2 + (3q^2+5q+3) t + 1$ \\
\vspace{0.1cm} & $(1,1,1)$ & $q^3 t^3 + (7q^3 +11q^2 +5q) t^2 + (5q^2 +11q +7) t + 1$ \\
$4$ & $(4)$ & $q^6 t^4 + (q^6 +q^5 +q^4 +q^3) t^3 + (q^5 +q^4 +2 q^3  +q^2 +q) t^2$ \\
 & & $+ (q^3 +q^2  +q +1) t + 1$ \\
 & $(1,3)$ & $q^6 t^4 + (3q^6 +5q^5 +5q^4 +3q^3) t^3 + (3q^5 +7q^4 +10q^3 +7q^2 +3q) t^2$\\
 & & $+ (3 q^3 +5 q^2 +5 q +3) t + 1$ \\
 & $(2,2)$ & $q^6 t^4 + (3q^6 +7q^5 +7q^4 +3q^3) t^3 + (5q^5 +13q^4 +18q^3 +13q^2 +5q) t^2 $\\
 & & $+ (3q^3 +7q^2 +7q +3) t + 1$ \\
 & $(1,1,2)$ & $q^6t^4 +(7q^6 +15q^5 +13q^4 +5q^3) t^3 + (9q^5 +27q^4 + 38q^3 +27q^2+ 9q) t^2 $ \\
 & & $ +(5q^3 +13q^2 +15q +7) t + 1$ \\
 & $(1,1,1,1)$ & $q^6 t^4 + (15q^6 +31q^5 +23q^4 +7q^3) t^3 + (17q^5 +57q^4 +82q^3 +57q^2 +17q) t^2 $\\
\vspace{0.1cm} & & $ + (7q^3 +23q^2 +31q +15)t + 1$ \\
 $5$ & $(5)$ & $q^{10} t^5 +(q^{10} +q^9 +q^8 +q^7 +q^6) t^4$\\
 & & $ +(q^9 +q^8 +2q^7 +2q^6 +2q^5 +q^4 +q^3) t^3 $\\
 & & $+ (q^7 +q^6 +2q^5 +2q^4 +2q^3 +q^2 +q) t^2$ \\
 & & $ + (q^4 +q^3 +q^2 +q +1) t + 1$ \\ 
 & $(1,4)$ & $q^{10} t^5 +(3q^{10} +5q^9 +5q^8 +5q^7 +3q^6) t^4$\\
 & & $ + (3q^9 +7q^8 +12q^7 +14q^6 +12q^5 +7q^4 +3q^3) t^3$\\
 & & $ +(3q^7 +7q^6 +12q^5 +14q^4 +12q^3 +7q^2 +3q) t^2 $\\
 & & $+ (3q^4 +5q^3 +5q^2 +5q +3) t + 1$ \\
 & $(2,3)$ & $q^{10} t^5 +(3q^{10} +7q^9 +9q^8 +7q^7  +3q^6) t^4 $\\
 & & $+ (5q^9 +15q^8 +28q^7 +34q^6 +28q^5 +15q^4 +5q^3) t^3 $\\
 & & $+ (5q^7 +15q^6 +28q^5 +34q^4 +28q^3 +15q^2 +5q) t^2 $\\
 & & $+ (3q^4 +7q^3 +9q^2 +7q +3) t + 1$ \\
 & $(1,1,3)$ & $q^{10} t^5 + (7q^{10} +15q^9 +17q^8 +13q^7 +5q^6) t^4 $\\
 & & $ + (9q^9 +31q^8 +58q^7 +70q^6 +56q^5 +29q^4 +9q^3) t^3$ \\
 & & $ + (9q^7 +29q^6 +56q^5 +70q^4 +58q^3 +31q^2 +9q) t^2$\\
 & & $ + (5q^4 +13q^3 +17q^2 +15q +7) t + 1$ \\
 & $(1,2,2)$ & $q^{10} t^5 + (7q^{10} +19q^9 +23q^8 +15q^7 +5q^6) t^4 $\\
 & & $+ (13q^9 +49q^8 +94q^7 +112q^6 +88q^5 +43q^4 +11q^3) t^3$ \\
 & & $ + (11q^7 +43q^6 +88q^5 +112q^4 +94q^3 +49q^2 +13q) t^2 $\\
 & & $ + (5q^4 +15q^3 +23q^2 +19q +7) t + 1$ \\
 & $(1,1,1,2)$ & $q^{10} t^5 + (15q^{10} +39q^9 +43q^8 +25q^7 +7q^6) t^4 $\\
 & & $ + (25q^9 + 101q^8 +196q^7 +232q^6 +176q^5 +81q^4 +19q^3) t^3$ \\
 & & $ + (19q^7 +81q^6 +176q^5 +232q^4 +196q^3 +101q^2 +25q) t^2 $\\
 & & $+ (7q^4 +25q^3  +43q^2 +39q +15) t + 1$ \\
 & $(1,1,1,1,1)$ & $q^{10} t^5 + (31q^{10} +79q^9 +79q^8 +39q^7 +9q^6) t^4$\\
 & & $ + (49q^9 +209q^8 +410q^7 +480q^6 +352q^5 +151q^4 +31q^3) t^3$\\
 & & $ + (31q^7 +151q^6 +352q^5 +480q^4 +410q^3 +209q^2 +49q) t^2$\\
 & & $ + (9q^4 +39q^3 +79q^2 +79q +31) t + 1$ 
 \end{longtable}
 \normalsize
 \vspace{-0.3cm}

\end{document}